\documentclass[12pt]{article}
\usepackage{amsmath,amssymb,amsthm,amsfonts,latexsym,hyperref,enumerate}

\newtheorem{lemma}{Lemma}[section]
\newtheorem{theorem}[lemma]{Theorem}

\newtheorem{proposition}[lemma]{Proposition}

\begin{document}

\title{A family of Neumaier graphs containing examples with exactly five eigenvalues}
\author{Bart De Bruyn, Rhys J. Evans, Sergey Goryainov and Jack Koolen}
\maketitle{}

\begin{abstract}
A Neumaier graph is an edge-regular graph with a regular clique. Such a graph is said to have parameters $(v,k,\lambda;e,s)$ if it is a $k$-regular graph on $v$ vertices having a clique of size $s$ such that every edge is contained in $\lambda$ triangles and every vertex outside $C$ is adjacent with exactly $e$ vertices inside $C$. It was an open problem whether Neumaier graphs can exist with exactly five eigenvalues. In the present paper, we describe a family of Neumaier graphs, and show that inside this family there are 1063 nonisomorphic Neumaier graphs with parameters $(v,k,\lambda;e,s)=(48,14,2;1,4)$, among which 25 have exactly five eigenvalues. These 1063 graphs are also the first known examples of Neumaier graphs for the mentioned parameters. 
\end{abstract}

\section{Introduction} \label{sec1}

For $k \geq 1$, a $k$-regular graph on $v \geq 2$ vertices is called {\em edge-regular} if any two adjacent vertices have a constant number $\lambda$ of common neighbours. A nonempty clique $C$ of a regular graph is called {\em regular} if every vertex outside $C$ is adjacent to a constant number $e$ of vertices in $C$, called the {\em nexus} of $C$. An edge-regular graph with parameters $(v,k,\lambda)$ having a regular clique $C$ of size $s$ and nexus $e$ called a {\em Neumaier graph} with parameters $(v,k,\lambda;e,s)$. 

This terminology stems from a paper of Neumaier \cite{Ne} in which he studies regular cliques in edge-regular graphs, and a class of designs whose collinearity graphs have regular cliques. All the constructed examples were strongly regular, and Neumaier wondered whether there could exist edge-regular non-strongly regular graphs with a regular clique (\cite[p. 248]{Ne}). Such non-strongly regular graphs are now called {\em strictly Neumaier graphs}.

The problem had been open for quite some time, until Greaves and Koolen \cite{Gr-Ko:1} constructed the first infinite family of strictly Neumaier graphs using cyclotomic numbers. In the mean time, many more examples and families have been found \cite{AA-WC-MDB-JK-SZ,AA-MDB-SZ,RE-SG-EK-AM,RE,Go-Sh,Gr-Ko:2,Gr-ZT}, the smallest one having 16 vertices \cite{RE} and being unique up to isomorphisms \cite{AA-MDB-SZ}. This smallest example has parameters $(16,9,4;2,4)$ and exactly six distinct eigenvalues. The second smallest strictly Neumaier graphs all have parameters $(24,8,2;1,4)$, and among the six known examples, there are four that also have six eigenvalues \cite[Corollary 3]{RE-SG-EK-AM}.  

As no strictly Neumaier graphs were known with less than six eigenvalues, the open problem arose whether examples with exactly four or five eigenvalues can exist (those with less than four are strongly regular). The case with exactly four eigenvalues was settled in \cite{AA-BDB-JD-JK}, where it was shown that such strictly Neumaier graphs cannot exist. The case with exactly five eigenvalues remained open till now.  

Indeed, in the current paper, we will construct 25 Neumaier graphs with exactly five eigenvalues. In fact, we will construct a family of Neumaier graphs, depending on two parameters $\lambda,\mu \in \mathbb{N}^\ast$ satisfying $\mu \mid (\lambda+2)$. Subsequently, we will classify all examples for $(\lambda,\mu) \in \{ (1,1),(1,3),(2,1),(2,2) \}$. For $(\lambda,\mu)$ equal to (1,1) or (1,3), the construction gives unique (known) strongly-regular Neumaier graphs, respectively being the $3 \times 3$ rook graph and the complement of the Sch\"afli graph. We also show that there are no examples for $(\lambda,\mu)=(2,1)$. The main bulk of work concerns the case $(\lambda,\mu)=(2,2)$, where we find up to isomorphisms 1063 examples, among which 25 have exactly five eigenvalues. To some extent we have hereby relied on computer computations. After constructing 319488 Neumaier graphs in an entirely theoretical way, we have performed computations in GAP \cite{GAP} and SageMath \cite{SAGE} to determine how many nonisomorphic ones there are. We also determined the eigenvalues by computer, but in view of the importance of the previously mentioned open problem whether Neumaier graphs with five eigenvalues can exist, we also provide a theoretical argument that allows to determine the spectrum for all these 25 Neumaier graphs in our classification.   

The paper is organized as follows. In Section \ref{sec2}, we provide a new construction method for Neumaier graphs. The basic ingredients here are three edge-regular graphs (whose parameters depend on the two mentioned numbers $\lambda$ and $\mu$), each with a suitable pair of partitions in independent sets. In Section \ref{sec3}, we discuss the case $\mu=1$, in particular the cases $(\lambda,\mu)=(1,1)$ and $(\lambda,\mu)=(2,1)$, as well as the case $(\lambda,\mu)=(1,3)$. In Section \ref{sec4}, we classify all examples for $(\lambda,\mu)=(2,2)$. The 1063 nonisomorphic examples there can be divided into 109 subclasses that can be found in Tables \ref{tab2} and \ref{tab3}. These tables also contain basic information such as the class sizes and the spectra. Families 1, 4 and 34 are those that contain the 25 graphs with exactly five eigenvalues. In fact, those families are exactly those for which the three involved edge-regular graphs are all $4 \times 4$ rook graphs or Shrikhande graphs, with at least one $4 \times 4$ rook graph (otherwise, you still get edge-regular graphs with parameters $(48,14,2)$, but without regular cliques).  In Section \ref{sec5}, we prove in a theoretical way that any edge-regular graph obtained by combining three $4 \times 4$ rook graphs and/or Shrikhande graphs as described in the construction must have spectrum $14^1 6^1 2^{26} (-2)^{12} (-6)^8$.  
 
We have also verified by computer that there are only two vertex-transitive graphs among the 1063 graphs. Both are contained in Family 1, and only one of them is also a Cayley graph. In fact, this particular graph is a Cayley graph for two groups, namely $C_2 \times S_4$ and $C_2 \times C_2 \times A_4$.

We also provide an online document \cite{BDB-RE-SG-JK} that contains the g6-strings for the 1063 constructed graphs, ordered according to which of the 109 subclasses they are contained in.

\section{A new family of Neumaier graphs} \label{sec2}

Let $\lambda$ and $\mu$ be positive integers such that $\mu$ is a divisor of $\lambda+2$. Put $n:=\mu \lambda (\lambda+2)$, $k := \lambda(\mu+1)$ and $\nu := \frac{\lambda+2}{\mu}$.

Suppose $\Gamma_1=(V_1,E_1)$, $\Gamma_2=(V_2,E_2)$ and $\Gamma_3=(V_3,E_3)$ are three vertex-disjoint edge-regular graphs with parameters $(n,k,\lambda)$.

For $i \in \{ 1,2,3 \}$, let $\Omega_i$ be a partition of $V_i$ in $\nu$ sets of size $\mu^2 \lambda$, and let $\mathcal{P}_i,\mathcal{P}_i'$ be two partitions of $V_i$ in $\mu \nu$ independent sets of size $\mu \lambda$ such that every element of $\Omega_i$ is the union of $\mu$ elements of $\mathcal{P}_i$ and of $\mu$ elements of $\mathcal{P}_i'$. Suppose also that the following holds:
\begin{quote}
If for some $i \in \{ 1,2,3 \}$, we have $\omega \in \Omega_i$, $P \in \mathcal{P}_i$ and $P' \in \mathcal{P}_i'$ such that $P,P' \subseteq \omega$, then $|P \cap P'|=\lambda$.
\end{quote}
Suppose $i,j \in \{ 1,2,3 \}$ are distinct. A bijection $\varphi:\mathcal{P}_i \to \mathcal{P}_j'$ is called {\em admissible} if for every $\omega = P_1 \cup P_2 \cup \cdots \cup P_\mu \in \Omega_i$ with $P_1,P_2,\ldots,P_\mu \in \mathcal{P}_i$, we have that $\omega^\varphi := P_1^\varphi \cup P_2^\varphi \cup \cdots \cup P_\mu^\varphi \in \Omega_j$. A bijection $\varphi:\mathcal{P}_{j}' \to \mathcal{P}_i$ is called {\em admissible} if its invers $\varphi^{-1}:\mathcal{P}_i \to \mathcal{P}_j'$ is admissible.

Consider now three admissible bijections $\varphi_{12}: \mathcal{P}_1 \to \mathcal{P}_2'$, $\varphi_{23}:\mathcal{P}_2 \to \mathcal{P}_3'$ and $\varphi_{31}:\mathcal{P}_3 \to \mathcal{P}_1'$ such that $\varphi_{31} \circ \varphi_{23} \circ \varphi_{12}$ induces the trivial permutation on the set $\Omega_1$. The inverses of $\varphi_{12}$, $\varphi_{23}$ and $\varphi_{31}$ are respectively denoted by $\varphi_{21}$, $\varphi_{32}$ and $\varphi_{13}$.

We now define a graph $\overline{\Gamma}$ with vertex set $V_1 \cup V_2 \cup V_3$. Adjacency is defined as follows:
\begin{itemize}
\item two vertices of the same $V_i$, $i \in \{ 1,2,3 \}$, are adjacent in $\overline{\Gamma}$ if and only if they are adjacent in $\Gamma_i$;
\item let $i,j \in \{ 1,2,3 \}$ with $i \not= j$, let $\mathcal{P}^\ast \in \{ \mathcal{P}_i,\mathcal{P}_i' \}$ and $\mathcal{Q}^\ast \in \{ \mathcal{P}_j,\mathcal{P}_j' \}$ such that $\varphi_{ij}$ is a bijection between $\mathcal{P}^\ast$ to $\mathcal{Q}^\ast$; a vertex $v \in V_i$ is then adjacent in $\overline{\Gamma}$ with a vertex $w \in V_j$ if $w \in \varphi_{ij}(P)$, where $P$ is the unique element of $\mathcal{P}^\ast$ containing $v$. 
\end{itemize}
Then the following holds.

\begin{theorem} \label{theo2.1}
The graph $\overline{\Gamma}$ satisfies the following properties:
\begin{enumerate}
\item[$(i)$] There are $3 \mu \lambda (\lambda+2)$ vertices.
\item[$(ii)$] Every vertex is adjacent to $(3\mu+1)\lambda$ other vertices.
\item[$(iii)$] Every two adjacent vertices have exactly $\lambda$ common neighbors.
\item[$(iv)$] If $\Gamma_i$ for some $i \in \{ 1,2,3 \}$ contains a clique $C$ of size $\lambda+2$, then in $\overline{\Gamma}$, $C$ is a regular clique of size $\lambda+2$ and nexus $1$, implying that $\overline{\Gamma}$ is then a Neumaier graph with parameters $(3 \mu \lambda (\lambda+2),(3\mu+1)\lambda,\lambda;1,\lambda+2)$.
\end{enumerate}
\end{theorem}
\begin{proof}
(i) The number of vertices of $\overline{\Gamma}$ is equal to $|V_1 \cup V_2 \cup V_3| = 3\mu \lambda (\lambda+2)$.

\medskip (ii)  Consider a vertex of $\overline{\Gamma}$ belonging to $V_i$, $i \in \{ 1,2,3 \}$. Then $x$ is adjacent with $k=\lambda(\mu+1)$ vertices of $V_i$, namely the $k$ vertices of $\Gamma_i$ adjacent to $x$. We now also show that $x$ is adjacent to exactly $\lambda\mu$ vertices of each $V_j$, $j \in \{ 1,2,3 \} \setminus \{ i \}$, resulting all together in $(3\mu+1)\lambda$ neighbors of $x$. Let $\mathcal{P}^\ast \in \{ \mathcal{P}_i,\mathcal{P}_i' \}$ and $\mathcal{Q}^\ast \in \{ \mathcal{P}_j,\mathcal{P}_j' \}$ such that $\varphi_{ij}: \mathcal{P}^\ast \to \mathcal{Q}^\ast$. Then the neighbors of $x$ in $V_j$ are exactly the $\mu \lambda$ elements in $\varphi_{ij}(P)$, where $P$ is the unique element of $\mathcal{P}^\ast$ containing $x$.

\medskip (iii) Let $\{ x_1,x_2 \}$ be an edge of $\overline{\Gamma}$. We distinguish two cases:

(a) Suppose $x_1$ and $x_2$ belong to the same $V_i$, $i \in \{ 1,2,3 \}$. Then there are exactly $\lambda$ vertices in $V_i$ adjacent to $x_1$ and $x_2$, namely the $\lambda$ neighbors of $x_1$ and $x_2$ in $\Gamma_i$. We show that these are the only neighbors of $x_1$ and $x_2$ in $\overline{\Gamma}$. Suppose to the contrary that the vertex $y$ of $V_j$, $j \in \{ 1,2,3 \} \setminus \{ i \}$, is adjacent to both $x_1$ and $x_2$. Let $\mathcal{P}^\ast \in \{ \mathcal{P}_i,\mathcal{P}_i' \}$ and $\mathcal{Q}^\ast \in \{ \mathcal{P}_j,\mathcal{P}_j' \}$ be such that $\varphi_{ij}:\mathcal{P}^\ast \to \mathcal{Q}^\ast$. Let $P_1$ and $P_2$ be the respective elements of $\mathcal{P}^\ast$ containing $x_1$ and $x_2$. Since $P_1$ and $P_2$ are independent sets and $x_1 \sim x_2$, we know that $P_1$ and $P_2$ are disjoint, and so also $\varphi_{ij}(P_1)$ and $\varphi_{ij}(P_2)$ are disjoint. However, this is impossible: as $x_l \sim y$, we have $y \in \varphi_{ij}(P_l)$ for every $l \in \{ 1,2 \}$.  

(b) Suppose $x_1$ and $x_2$ are two vertices belonging to distinct $V_i$'s, say $x_1 \in V_{i_1}$ and $x_2 \in V_{i_2}$, where $i_1,i_2 \in \{ 1,2,3 \}$ with $i_1 \not= i_2$. Put $\{ i_3 \} := \{ 1,2,3 \} \setminus \{ i_1,i_2 \}$. We show that there are no neighbors of $x_1$ and $x_2$ belonging to $V_{i_1} \cup V_{i_2}$. By symmetry, it suffices to show this for $V_{i_2}$. Let $\mathcal{P}^\ast \in \{ \mathcal{P}_{i_1},\mathcal{P}_{i_1}' \}$ and $\mathcal{Q}^\ast \in \{ \mathcal{P}_{i_2},\mathcal{P}_{i_2}' \}$ such that $\varphi_{i_1i_2}:\mathcal{P}^\ast \to \mathcal{Q}^\ast$. The vertices of $V_{i_2}$ adjacent to $x_1$ are exactly the elements of the set $\varphi_{i_1i_2}(P)$, where $P$ is the unique element of $\mathcal{P}^\ast$ containing $x_1$. Since $x_1 \sim x_2$, we have $x_2 \in \varphi_{ij}(P)$. Now, $\varphi_{ij}(P)$ is an independent set and so no vertex of $\varphi_{ij}(P)$ distinct from $x_2$ is adjacent to $x_2$. So, no common neighbor of $x_1$ and $x_2$ belongs to $V_{i_2}$ (or $V_{i_1}$).

So, the only possible neighbors of $x_1$ and $x_2$ belong to $V_{i_3}$. Let $\mathcal{P}^\ast \in \{ \mathcal{P}_{i_1},\mathcal{P}_{i_1}' \}$, $\mathcal{Q}^\ast \in \{ \mathcal{P}_{i_2},\mathcal{P}_{i_2}' \}$ and $\mathcal{R}^\ast,\mathcal{S}^\ast \in \{ \mathcal{P}_{i_3},\mathcal{P}_{i_3}' \}$ with $\mathcal{R}^\ast \not= \mathcal{S}^\ast$ such that $\varphi_{i_1i_3}:\mathcal{P}^\ast \to \mathcal{R}^\ast$, $\varphi_{i_2i_3}:\mathcal{Q}^\ast \to \mathcal{S}^\ast$. Let $P_1$ and $P_2$ be the unique elements in respectively $\mathcal{P}^\ast$ and $\mathcal{Q}^\ast$ containing $x_1$ and $x_2$. Let $\omega_1 \in \Omega_1$ and $\omega_2 \in \Omega_2$ such that $P_1 \subseteq \omega_1$ and $P_2 \subseteq \omega_2$. Since $x_1 \in P_1$ and $x_2 \sim x_1$, we have $x_2 \in \varphi_{i_1i_2}(P_1)$. Since $\varphi_{i_1i_2}(P_1)$ is contained in $\omega_1^{\varphi_{i_1i_2}}$, we have $\omega_2 = \omega_1^{\varphi_{i_1i_2}}$. Since $\varphi_{i_3 i_1} \circ \varphi_{i_2 i_3} \circ \varphi_{i_1 i_2}$ is the trivial permutation of $\Omega_{i_1}$, we have $\omega_3 := \omega_1^{\varphi_{i_1i_3}} = \omega_2^{\varphi_{i_2i_3}}$. The common neighbors of $x_1$ and $x_2$ (belonging to $V_3$) are the vertices of the set $\varphi_{i_1 i_3}(P_1) \cap \varphi_{i_2 i_3}(P_2)$. As $\varphi_{i_1 i_3}(P_1),\varphi_{i_2 i_3}(P_2) \subseteq \omega_3$, $\varphi_{i_1 i_3}(P_1) \in \mathcal{R}^\ast$, $\varphi_{i_2 i_3}(P_2) \in \mathcal{S}^\ast$ and $\{ \mathcal{R}^\ast,\mathcal{S}^\ast \} = \{ \mathcal{P}_{i_3},\mathcal{P}_{i_3}' \}$, we have $|\varphi_{i_1 i_3}(P_1) \cap \varphi_{i_2 i_3}(P_2)|=\lambda$, and so $x_1$ and $x_2$ have exactly $\lambda$ common neighbours.

\medskip (iv) If a vertex $y$ of $\Gamma_i$ outside $C$ is adjacent to two distinct vertices $x_1$ and $x_2$ of $C$, then $x_1$ and $x_2$ would have at least $\lambda+1$ common neighbors, namely the vertices inside $(C \setminus \{ x_1,x_2 \}) \cup \{ y \}$, a contradiction.

So, the number of vertices of $\Gamma_i$ adjacent to a necessarily unique vertex of $C$ is equal to $|C| \cdot (k - |C| + 1) = (\lambda+2)(\lambda \mu + \lambda - \lambda - 2 +1) = (\mu \lambda -1) (\lambda+2) = n - |C|$. These are all the vertices of $V_i \setminus C$. So, $C$ is a regular clique with nexus $1$ in the graph $\Gamma_i$. We also show that $C$ is a regular clique with nexus $1$ inside $\overline{\Gamma}$. To that end, consider a vertex $x$ belonging to $V_j$, $j \in \{ 1,2,3 \} \setminus \{ i \}$. Let $\mathcal{P}^\ast \in \{ \mathcal{P}_j,\mathcal{P}_j' \}$ and $\mathcal{Q}^\ast \in \{ \mathcal{P}_i,\mathcal{P}_i' \}$ such that $\varphi_{ji}: \mathcal{P}^\ast \to \mathcal{Q}^\ast$. The $\mu \nu = \lambda + 2$ independent sets of $\mathcal{Q}^\ast$ intersect $C$ in at most $1$ vertex, but as $\mathcal{Q}^\ast$ partitions $V_i$, they must intersect $C$ in exactly one vertex. The vertices of $C$ adjacent to $x$ are the vertices of $\varphi_{ji}(P) \cap C$, where $P$ is the unique element of $\mathcal{P}^\ast$ containing $x$. As $\varphi_{ji}(P) \in \mathcal{Q}^\ast$, we have $|\varphi_{ji}(P) \cap C|=1$.
\end{proof}

\medskip \noindent In Sections \ref{sec3} and \ref{sec4}, we treat some instances where the parameters $\lambda$ and $\mu$ are relatively small. These cases are also interesting for the following proposition. As regard to this result, we recall that there are up to isomorphisms two strongly regular graphs with parameters $(16,6,2,2)$, namely the $4 \times 4$ rook graph and the Shrikhande graph \cite{Sh}.

\begin{proposition} \label{prop2.2}
Suppose $\Gamma_i$ with $i \in \{ 1,2,3 \}$ is a strongly regular graph and contains a clique $C$ of size $\lambda+2$. Then one of the following cases occurs:
\begin{itemize}
\item $(\lambda,\mu)=(1,1)$ and $\Gamma_i \cong K_3$ (regarded as an imprimitive strongly regular graph);
\item $(\lambda,\mu)=(1,3)$ and $\Gamma_i$ is the $3 \times 3$ rook graph;
\item $(\lambda,\mu)=(2,2)$ and $\Gamma_i$ is either the $4 \times 4$ rook graph or the Shrikhande graph.
\end{itemize}
\end{proposition}
\begin{proof}
Suppose first that $\lambda \in \{ 1,2 \}$. The fact that $\mu \mid (\lambda+2)$ then implies that $(\lambda,\mu) \in \{ (1,1),(1,3),(2,1),(2,2),(2,4) \}$. If $(\lambda,\mu)=(1,1)$, then $(n,k,\lambda)=(3,2,1)$ and $\Gamma_i$ is necessarily isomorphic to $K_3$ regarded as an imprimitive strongly regular graph. If $(\lambda,\mu)=(1,3)$, then $(n,k,\lambda)=(9,4,1)$ and $\Gamma_i$ is necessarily isomorphic to the $3 \times 3$ rook graph. If $(\lambda,\mu)$ is equal to $(2,1)$, $(2,2)$ or $(2,4)$, then $(n,k,\lambda)$ is equal to either $(8,4,2)$, $(16,6,2)$ or $(32,10,2)$, and the number of common neighbours that two distinct nonadjacent vertices have is respectively equal to $\frac{4 \cdot (4-2-1)}{8-4-1} = \frac{4}{3}$, $\frac{6 \cdot (6-2-1)}{16-6-1}=2$ and $\frac{10 \cdot (10-2-1)}{32-10-1} = \frac{10}{3}$. We conclude in this case that $(\lambda,\mu)=(2,2)$ and that $\Gamma_i$ is a strongly regular graph with parameters $(16,6,2,2)$, necessarily isomorphic to the $4 \times 4$ rook graph or the Shrikhande graph. 

Suppose that $\Gamma_i$ is a conference graph with $\lambda \geq 3$ (\cite[Section 1.3]{BCN}). Then $k=2(\lambda+1)$ and $n=2k+1$. The former equation implies that $\lambda(\mu+1) = 2(\lambda+1)$, or equivalently $\lambda \mu = \lambda+2$. The latter equation implies that $\mu \lambda (\lambda+2) = 2 \lambda (\mu+1) +1$, i.e. $2\lambda+1 = \mu \lambda^2 = \lambda (\lambda+2)$, or equivalently, $\lambda^2=1$, a contradiction.

So, suppose that $\lambda \geq 3$ and that $\Gamma_i$ is not a conference graph. Then the smallest eigenvalue $-m$ of $\Gamma_i$ must be integral by \cite[Theorem 1.3.1(ii)]{BCN}. As $C$ is a regular clique in $\Gamma_i$, we know from Hoffman's bound (\cite[Proposition 1.3.2]{BCN}) that $|C| = 1 + \frac{k}{m}$, or equivalently, $m = \frac{\lambda(\mu+1)}{\lambda+1}$. So, we must have $(\lambda+1) \mid (\mu + 1)$ and $\mu \geq \lambda$. As $\mu \mid (\lambda+2)$, we even have $\lambda \leq \mu \leq \lambda +2$ and as $1 \leq \frac{\lambda+2}{\mu} \leq \frac{\lambda+2}{\lambda} = 1 + \frac{2}{\lambda}$ with $\lambda \geq 3$, we have $\mu = \lambda+2$. But then $(\lambda+1) \mid (\mu+1)=(\lambda+3)$ implies $\lambda+1 \mid 2$, again a contradiction.
\end{proof}

\medskip \noindent \textbf{Remark.} In case $\Gamma_i$ is a strongly regular graph, then \cite[Proposition 1.3.2]{BCN} also implies that the independent sets in $\mathcal{P}_i$ and $\mathcal{P}_i'$ are regular. For the $4 \times 4$ rook graphs and the Shrikhande graphs, these independent sets are regular with nexus 2. For the graphs under consideration in Section \ref{sec5}, the reasoning allowing to determine the spectrum relies on the fact that the $\Gamma_i$'s are strongly regular and that the independent sets in each $\mathcal{P}_i \cup \mathcal{P}_i'$ are regular. 

\section{Discussion of the cases $(\mu,\lambda) \in \{ (1,1),(1,3),(2,1) \}$} \label{sec3}

Consider the construction of Section \ref{sec2}, and suppose that $\mu=1$. Then the independent sets have size $\lambda$ and therefore the condition $|P \cap P'|=\lambda$ from Section \ref{sec2} implies that $P=P'$ and thus that $\mathcal{P}_i = \mathcal{P}_i'$ for every $i \in \{ 1,2,3 \}$. We have therefore the following situation.

The graphs $\Gamma_1=(V_1,E_1)$, $\Gamma_2=(V_2,E_2)$ and $\Gamma_3=(V_3,E_3)$ are three vertex-disjoint edge-regular graphs with parameters $(n,k,\lambda)=(\lambda(\lambda+2),2\lambda,\lambda)$. Each $\Gamma_i$ has a partition $\mathcal{P}_i = \{ A^{(i)}_1,A^{(i)}_2,\ldots,A^{(i)}_{\lambda+2} \}$ in $\lambda+2$ independent sets of size $\lambda$. The maps $\varphi_{12}$, $\varphi_{23}$, $\varphi_{31}$, $\varphi_{21} = \varphi_{12}^{-1}$, $\varphi_{32} = \varphi_{23}^{-1}$ and $\varphi_{13}=\varphi_{31}^{-1}$ can be regarded as permutations of the set $\{ 1,2,\ldots,\lambda+2 \}$ such that $\varphi_{31} \circ \varphi_{23} \circ \varphi_{12}$ is trivial. The vertex set of $\overline{\Gamma}$ is $V_1 \cup V_2 \cup V_3$. Two vertices of the same $V_i$ are adjacent in $\overline{\Gamma}$ whenever they are adjacent in $\Gamma_i$, and vertices $x \in V_i$ and $y \in V_j$ with $i \not= j$ are adjacent whenever $y \in A^{(j)}_{\varphi_{ij}(l)}$ if $A^{(i)}_l$ is the unique element of $\mathcal{P}_i$ containing $x$. 

Consider now the special case $(\lambda,\mu)=(1,1)$. Then $\Gamma_1$, $\Gamma_2$ and $\Gamma_3$ are edge-regular graphs with parameters $(n,k,\lambda)=(3,2,1)$ and therefore isomorphic to $K_3$. In this case $\overline{\Gamma}$ is the $3 \times 3$ rook graph.

We prove that the case $(\lambda,\mu)=(2,1)$ cannot occur by showing that there are no edge-regular graphs $\Gamma$ with parameters $(n,k,\lambda)=(8,4,2)$. Note that any local graph of $\Gamma$ is a 2-regular graph on 4 vertices and therefore isomorphic to $C_4$. Take a vertex $x$ and denote by $x_1,x_2,x_3,x_4$ the four neighbours of $x$ such that $x_1 \sim x_2 \sim x_3 \sim x_4 \sim x_1$. 

We show that any three consecutive vertices of the local graph $\Gamma_x$, say $x_1$, $x_2$ and $x_3$, have besides $x$ a second common neighbour. Denote by $\widetilde{x}$ the fourth neighbour of $x_2$ besides $x$, $x_1$ and $x_3$. The vertices $\widetilde{x}$, $x_1$ and $x_3$ are three vertices of the local graph $\Gamma_{x_2}$ which is isomorphic to $C_4$. As $x_1$ and $x_3$ are not adjacent, the vertex $\widetilde{x}$ needs to be adjacent to both $x_1$ and $x_3$. 

Applying the above to any three consecutive vertices of the local graph $\Gamma_x$, we see that the vertex $\widetilde{x}$ needs to be adjacent to $x_1$, $x_2$, $x_3$ and $x_4$. But then the graph induced on $\{ x,x_1,x_2,x_3,x_4,\tilde{x} \}$ is the octahedron graph, which is 4-regular on 6 vertices. This graph cannot be extended to a 4-regular graph on 8 vertices.

\medskip We now discuss the case $(\lambda,\mu)=(1,3)$. 

Consider first the unique generalized quadrangle $\mathcal{Q}$ of order $(2,4)$ as defined in \cite{SP-JAT}. This generalized quadrangle has the property that every line is incident with exactly three points, every point is incident with exactly five lines and every two noncollinear points have exactly five common neighbours. For every point $x$ of $\mathcal{Q}$ we denote by $x^\perp$ the set of points of $\mathcal{Q}$ collinear or equal to $x$. The generalized quadrangle $\mathcal{Q}$ is known to have subgeometries that are $3 \times 3$ subgrids. The collinearity graphs of such subgrids are $3 \times 3$ rook graphs.

A $3 \times 3$ subgrid $G$ of $\mathcal{Q}$ has six {\em ovoids}. These are sets consisting of three mutually noncollinear points, i.e. they are independent sets of size 3 in the corresponding $3 \times 3$ rook graph. $G$ has two partitions $\mathcal{P}$ and $\mathcal{P}'$ in ovoids and it holds that $|O \cap O'|=1$ for all $(O,O') \in \mathcal{P} \times \mathcal{P}'$. If $x$ is a vertex outside $G$, then $x^\perp \cap G$ is an ovoid of $G$, the so-called {\em ovoid subtended} by $x$. Any two noncollinear points $x_1$ and $x_2$ of $G$ are contained in a unique ovoid of $G$, and they have five common neighbours among which two are contained in $G$ and three are outside $G$. This implies that every ovoid of $G$ is subtended by exactly three points outside $G$.

By \cite[Lemma 1.38]{BDB}, the generalized quadrangle $\mathcal{Q}$ has a partition $\{ G_1,G_2,G_3 \}$ in three $3 \times 3$ subgrids. Take such a subgrid $G_i$, $i \in \{ 1,2,3 \}$. If $L = \{ x_1,x_2,x_3 \}$ is a line disjoint from $G_i$, then $\{ x_1^\perp \cap G_i,x_2^\perp \cap G_i,x_3^\perp \cap G_i \}$ is a partition of $G_i$ in ovoids. Applying this to the lines of a subgrid $G_j$ with $j \in \{ 1,2,3 \} \setminus \{ i \}$, we see that all nine ovoids $x^\perp \cap G_i$ with $x \in G_j$ are part of the same partition $\mathcal{P}_i$ of $G_i$ in ovoids. The three points subtending a particular ovoid in $\mathcal{P}_i$ are therefore all contained in $G_j$. If $G_l$ is the remaining grid, then we see that the nine ovoids $x^\perp \cap G_i$ with $x \in G_l$ are part of the other partition $\mathcal{P}_i'$ of $G_i$ in ovoids.

The above implies that if we apply the construction of Section \ref{sec2} to three $3 \times 3$ rook graphs $\Gamma_1$, $\Gamma_2$ and $\Gamma_2$, then for a suitable choice of the partitions $\mathcal{P}_1$, $\mathcal{P}_1'$, $\mathcal{P}_2$, $\mathcal{P}_2'$, $\mathcal{P}_3$, $\mathcal{P}_3'$, and the maps $\varphi_{12}$, $\varphi_{23}$, $\varphi_{31}$, we can get a Neumaier graph with parameters (27,10,1;1,3), which is isomorphic to the collinearity graph of the unique generalized quadrangle of order $(2,4)$. This collinearity graph is the complement of the Sch\"afli graph.

Conversely, take the construction from Section \ref{sec2} in the case that $\lambda=1$ and $\mu=3$. Then $\Gamma_1$, $\Gamma_2$ and $\Gamma_3$ are edge-regular graphs with parameters $(n,k,\lambda)=(9,4,1)$, necessarily isomorphic to $3 \times 3$ rook graphs. Any Neumaier graph $\overline{\Gamma}$ arising from the construction must have parameters $(27,10,1;1,3)$. The maximal cliques have thus size $3$. Regarding these maximal cliques as the lines of a point-line geometry $\mathcal{Q}'$ with the vertices of $\overline{\Gamma}$ as points, we observe that $\mathcal{Q}'$ satisfies the following: (i) every line of $\mathcal{Q}'$ has three points and every point of $\mathcal{Q}'$ is contained in five lines; (ii) for every non-incident point-line pair $(x,L)$ of $\mathcal{Q}'$, there exists a unique point on $L$ collinear with $x$. $\mathcal{Q}'$ is therefore isomorphic to the unique generalized quadrangle of order $(2,4)$. $\overline{\Gamma}$ is the collinearity graph of $\mathcal{Q}$ and is therefore isomorphic to the complement of the Sch\"afli graph.

\section{Treatment of the case $(\lambda,\mu)=(2,2)$} \label{sec4}

Consider again the construction from Section \ref{sec2}, but this time with $\lambda=\mu=2$. Let $\Gamma = \Gamma_i$ for some $i \in \{ 1,2,3 \}$, and put $\Omega := \Omega_i$, $\mathcal{P} := \mathcal{P}_i$, $\mathcal{P}' := \mathcal{P}_i'$. Recall that $\Gamma$ is an edge-regular graph with parameters $(v,k,\lambda) = (\mu \lambda (\lambda+2),\lambda(\mu+1),\lambda)=(16,6,2)$. Put $\mathcal{P} = \{ P_1,P_2,P_3,P_4 \}$, $\mathcal{P}' = \{ P_1',P_2',P_3',P_4' \}$ and $\Omega = \{ \omega_1,\omega_2 \}$ with $\omega_1 = P_1 \cup P_2 = P_1' \cup P_2'$ and $\omega_2 = P_3 \cup P_4 = P_3' \cup P_4'$. As $|P_i \cap P_j'|=2$ for all $i,j \in \{ 1,2 \}$ and all $i,j \in \{ 3,4 \}$, we can label the vertices of $\Gamma$ by $v_1,v_2,\ldots,v_{16}$ such that 
\[   P_1 = \{ v_1,v_3,v_5,v_7  \},\quad P_2 = \{ v_2,v_4,v_6,v_8 \},\quad P_1' = \{ v_1,v_3,v_6,v_8 \}, \quad P_2' = \{ v_2,v_4,v_5,v_7 \}, \]
\[   P_3 = \{ v_9,v_{11},v_{13},v_{15} \},P_4 = \{ v_{10},v_{12},v_{14},v_{16} \},P_3' = \{ v_9,v_{11},v_{14},v_{16} \},P_4' = \{ v_{10},v_{12},v_{13},v_{15} \}. \]
Since $P_1,P_2,\ldots,P_3',P_4'$ are independent sets, the only possible edges between vertices belonging to the same $\omega_i$, $i \in \{ 1,2 \}$, are the 8 edges $\{ v_1,v_2 \}$, $\{ v_2,v_3 \}$, $\{ v_3,v_4 \}$, $\{ v_1,v_4 \}$, $\{ v_5,v_6 \}$, $\{ v_6,v_7 \}$, $\{ v_7,v_8 \}$, $\{ v_5,v_8 \}$ and the 8 edges $\{ v_9,v_{10} \}$, $\{ v_{10},v_{11} \}$, $\{ v_{11},v_{12} \}$, $\{ v_9,v_{12} \}$, $\{ v_{13},v_{14} \}$, $\{ v_{14},v_{15} \}$, $\{ v_{15},v_{16} \}$, $\{ v_{13},v_{16} \}$. So, there are most $(8 + 8) \cdot \lambda =32$ triangles in $\Gamma$ containing two vertices of the same $\omega_i$. Note also that each triangle in $\Gamma$ necessarily contains two vertices of the same $\omega_i$ and one vertex of the other $\omega_j$ in $\Omega$. All together, there are exactly $\frac{1}{6} vk\lambda = 32$ triangles in $\Gamma$, implying that each of the above 16 pairs are actually edges of $\Gamma$. This means that the triple $(\Gamma,\omega_1,\omega_2)$ satisfies the following properties:
\begin{enumerate}
\item[(A)] $\Gamma$ is an edge-regular graph with parameters $(v,k,\lambda)=(16,6,2)$;
\item[(B)] $\{ \omega_1,\omega_2 \}$ is a partition of the vertex set $V$ of $\Gamma$ in two sets of size 8 such that the induced subgraph on each $\omega_i$ is the disjoint union of two $C_4$'s.
\end{enumerate}
We now classify, up to graph isomorphisms, all triples $(\Gamma,\omega_1,\omega_2)$ that satisfy the above properties (A) and (B). We show that there are 12 such triples.

Since each local graph $\Gamma_x$, $x \in V$, is a 2-regular graph on 6 vertices, it is isomorphic to either $C_6$ or the union to two $K_3$'s. In the former case, $x$ is not contained in any 4-clique, and in the latter case, $x$ is contained in exactly two 4-cliques which meet in the singleton $\{ x \}$.

We show that if one local space $\Gamma_x$, $x \in V$, is the union of two disjoint $K_3$'s, then $\Gamma$ is the $4 \times 4$-rook graph. Let $L_1$ and $M_1$ denote the two 4-cliques through $x$. Every vertex of $L_1$ is contained in a 4-clique and hence is contained in exactly two $4$-cliques, resulting in four distinct 4-cliques $M_1$, $M_2$, $M_3$, $M_4$ that meet $L_1$ in singletons. Suppose two of these 4-cliques, say $M_{i_1}$ and $M_{i_2}$, meet in a vertex $u$, necessarily outside $L_1$. Then the two adjacent vertices in $L_1 \cap M_{i_1}$ and $L \cap M_{i_2}$ would have at least three common neighbors (namely, $u$ and two in $L_1$), a contradiction. So, $\{ M_1,M_2,M_3,M_4 \}$ form a partition of the vertex set in four 4-cliques. Similarly, there exists a partition $\{ L_1,L_2,L_3,L_4 \}$ of the vertex set in four 4-cliques that meet $M_1$ in singletons. These eight 4-cliques already cover $8 \cdot {4 \choose 2} = 48$ edges of $\Gamma$, i.e. all $\frac{1}{2}vk=48$ edges of $\Gamma$. As two distinct 4-cliques never meet in more than 1 vertex, we must have $|L_i \cap M_j| = 1$ for all $i,j \in \{ 1,2,3,4 \}$. It is now clear that $\Gamma$ must be a $4 \times 4$ rook graph. Note also the $4 \times 4$ rook graph has up to graph isomorphisms a unique ordered pair $(\omega_1,\omega_2)$ for which (B) holds. This already accounts for one of the 12 possible triples we alluded to. 

\medskip From now on, we assume that all local graphs are isomorphic to $C_6$. The subgraph induced on $\omega_i$, $i \in \{ 1,2 \}$, is the union of two 4-cycles $C_{ia}$ and $C_{ib}$. Put $\mathcal{C}= \{ C_{1a},C_{1b},C_{2a},C_{2b} \}$.

Consider a vertex $v$ of $\Gamma$ belonging to a certain $\omega_i$, $i \in \{ 1,2 \}$, and let $C$ denote the unique element of $\mathcal{C}$ containing $v$. Denote by $v'$ and $v''$ the two neighbors of $v$ in $C$, and by $C'$ and $C''$ the two elements of $\mathcal{C}$ contained in $\omega_{3-i}$. Note that $v'$ and $v''$ are nonadjacent in the local graph $\Gamma_v$. One of the following two cases then occurs.
\begin{itemize}
\item Suppose $v'$ and $v''$ have distance 2 in $\Gamma_v$. Then $v'$ and $v''$ have a unique neighbor $w$ in $\Gamma_v$ belonging to one of the two 4-cycles in $\omega_{3-i}$, say $C'$. The three vertices of $\Gamma_v$ distinct from $v'$, $v''$ and $w$ define three consecutive vertices in $\Gamma_v$, necessarily in the other 4-cycle $C''$ of $\omega_{3-i}$ as each of them is nonadjacent to $w$. The vertices $v$ and $v'$ have two neighbors, one of them is contained in $C'$ (namely $w$) and the other in $C''$. The same holds for $v$ and $v''$. There are two edges in $\Gamma_v$ not containing $v'$ nor $v''$. Both are contained in $C''$ (and none in $C'$).
\item Suppose $v'$ and $v''$ have distance 3 in $\Gamma_v$. There are two edges in $\Gamma_v$ not containing $v'$ nor $v''$, and vertices belonging to distinct edges cannot be adjacent, showing that one of these two edges is contained in $C'$, while the other is contained in $C''$ . The vertex $v$ has thus two neighbors in $C'$ and two neighbors in $C''$. The vertices $v$ and $v'$ have two neighbors, one of them is contained in $C'$ and the other in $C''$. The same holds for $v$ and $v''$. 
\end{itemize}
Let $C_1$ and $C_2$ be two elements of $\mathcal{C}$ belonging to distinct $\omega_i$'s. If $x$ and $y$ are two adjacent vertices belonging to some $C_i$, $i \in \{ 1,2 \}$, then by the above, $x$ and $y$ have a unique common neighbor in $C_{3-i}$.

Let $r$, $s$, $t$ and $u$ be two vertices of $C_1$ and suppose that $r \sim s \sim t \sim u \sim r$. Let $E_x$ with $x \in \{ r,s,t,u \}$ be the set of edges of $C_2$ that are in the neighborhood of $x$. By the above, we have $|E_x| \in \{ 0,1,2 \}$. If $x$ and $y$ are two adjacent vertices in $C_1$, then as $x$ and $y$ have a unique common neighbor in $C_2$, the union of the edges in $E_x$ and those in $E_y$ have a unique vertex in common if $|E_x|,|E_y| \geq 1$. As any two adjacent vertices in $C_2$ have a unique common neighbor in $C_1$, we have $E_x \cap E_y = \emptyset$ for distinct $x,y \in C_1$, and $E_r \cup E_s \cup E_t \cup E_u$ consists of the 4 edges of $C_2$. This leads to the following cases:
\begin{enumerate}
\item[(1)] $|E_r|=|E_s|=|E_t|=|E_u|=1$;
\item[(2)] two of $|E_r|$, $|E_s|$, $|E_t|$, $|E_u|$ are equal to 2, while the other two are equal to 0;
\item[(3)] one of $|E_r|$, $|E_s|$, $|E_t|$, $|E_u|$ is equal to 2, two are equal to 1, and one is equal on 0;  
\end{enumerate}
Suppose case (1) occurs. Then we can label the vertices of $C_2$ by $R$, $S$, $T$, $U$ such that $R \sim S \sim T \sim U \sim R$, $E_r = \{ RS \}$, $E_s = \{ ST \}$, $E_t = \{ TU \}$ and $E_u = \{ RU \}$.

Suppose case (2) occurs. Without loss of generality, we may assume that $|E_r|=2$. Then we can label the vertices of $C_2$ by $R$, $S$, $T$ and $U$ such that $R \sim S \sim T \sim U \sim R$, $E_r = \{ RS,ST \}$, $E_s = \emptyset$, $E_t = \{ TU,RU \}$ and $E_u = \emptyset$.

Suppose case (3) occurs. Without loss of generality, we may assume that $|E_r|=2$ and either $|E_u|=0$ or $|E_t|=0$. Then we can label the vertices of $C_2$ by $R$, $S$, $T$ and $U$ such that $R \sim S \sim T \sim U \sim R$ and either one of the following occurs:
\begin{enumerate}
\item[(3a)] $E_r = \{ RS,ST \}$, $E_s = \{ TU \}$, $E_t = \{ UR \}$ and $E_u = \emptyset$;
\item[(3b)] $E_r = \{ RS,ST \}$, $E_s = \{ TU \}$, $E_t = \emptyset$ and $E_u = \{ RU \}$.
\end{enumerate}
We say that $(C_1,C_2)$ of type (1), (2), (3a) or (3b) depending on which of the above cases occurs. Let $C_2' \in \mathcal{C}$ such that $C_2 \cup C_2' \in \Omega$. We now show that $(C_2,C_1)$ and $(C_1,C_2')$ have the same type as $(C_1,C_2)$.

We start with $(C_2,C_1)$. After reversing the roles of $C_1$ and $C_2$, we can define $E_R$, $E_S$, $E_T$ and $E_U$ in a similar way.

Suppose case (1) occurs. Then the edges between $C_1$ and $C_2$ are $\{ r,R \}$, $\{ r,S \}$, $\{ s,S \}$, $\{ s,T \}$, $\{ t,T \}$, $\{ t,U \}$, $\{ u,R \}$, $\{ u,U \}$, implying that $E_R = \{ ur \}$, $E_S = \{ rs \}$, $E_T = \{ st \}$ and $E_U = \{ ut \}$. So, also $(C_2,C_1)$ has type (1).

Suppose case (2) occurs. Then $s$ has a unique neighbor in $C_2$, which must also be a neighbor of $r$ and $t$, as $r$ and $s$ have a unique neighbor in $C_2$, as well as $s$ and $t$. This unique neighbor therefore must coincide with either $R$ or $T$, say $T$. In a similar way, the unique neighbor of $u$ in $C_2$ is equal to either $R$ or $T$. As this unique neighbor is adjacent with at most 3 vertices of $C_1$, it cannot be $T$ and hence must be $R$. The edges between $C_1$ and $C_2$ are therefore $\{ r,R \}$, $\{ r,S \}$, $\{ r,T \}$, $\{ s,T \}$, $\{ t,T \}$, $\{ t,U \}$, $\{ t,R \}$ and $\{ u,R \}$, implying that $E_R = \{ tu,ur \}$, $E_S = \emptyset$, $E_T = \{ rs,st \}$ and $E_U = \emptyset$. So, also $(C_2,C_1)$ has type (2).

Suppose case (3a) occurs. Similarly as in case (2), we can then argue that $R$ is the unique neighbor of $u$ inside $C_2$. The edges between $C_1$ and $C_2$ are therefore $\{ r,R \}$, $\{ r,S \}$, $\{ r,T \}$, $\{ s,T \}$, $\{ s,U \}$, $\{ t,U \}$, $\{ t,R \}$ and $\{ u,R \}$, implying that $E_R = \{ tu,ur \}$, $E_S = \emptyset$, $E_T = \{ rs \}$ and $E_U = \{ st \}$. So, also $(C_2,C_1)$ has type (3a).

Suppose case (3b) occurs. Similarly as in case (2), we can then argue that $U$ is the unique neighbor of $t$ inside $C_2$. The edges between $C_1$ and $C_2$ are therefore $\{ r,R \}$, $\{ r,S \}$, $\{ r,T \}$, $\{ s,T \}$, $\{ s,U \}$, $\{ t,U \}$, $\{ u,R \}$ and $\{ u,U \}$, implying that $E_R = \{ ur \}$, $E_S = \emptyset$, $E_T = \{ rs \}$ and $E_U = \{ st,tu \}$. So, also $(C_2,C_1)$ has type (3b).

In any case, $(C_2,C_1)$ has the same type as $(C_1,C_2)$. From the above, we also observe the following:
\begin{itemize}
\item if $(C_1,C_2)$ has type (1), then the values that $|\Gamma_1(x) \cap C_2|$ takes for $x$ ranging in $C_1$ are 2, 2, 2 and 2.
\item if $(C_1,C_2)$ has type (2), then the values that $|\Gamma_1(x) \cap C_2|$ takes for $x$ ranging in $C_1$ are 1, 1, 3 and 3.
\item if $(C_1,C_2)$ has type (3a), then the values that $|\Gamma_1(x) \cap C_2|$ takes for $x$ ranging in $C_1$ are 1, 2, 2 and 3, with the two vertices of $C_1$ giving rise to the values 1 and 3 being adjacent.
\item if $(C_1,C_2)$ has type (3b), then the values that $|\Gamma_1(x) \cap C_2|$ takes for $x$ ranging in $C_1$ are 1, 2, 2 and 3, with the two vertices of $C_1$ giving rise to the values 1 and 3 being nonadjacent.
\end{itemize}
As $|\Gamma_1(x) \cap C_2| + |\Gamma_1(x) \cap C_2'| = 4$ for every $x \in C_1$, we then see that $(C_1,C_2')$ must have the same type as $(C_1,C_2)$.

For ($t$) being equal to (1), (2), (3a) or (3b), we say that $(\omega_1,\omega_2)$ have type ($t$), if $(C_1,C_2)$ has type ($t$) for any $C_1 \in \mathcal{C}$ contained in $\omega_1$ and any $C_2 \in \mathcal{C}$ contained in $\omega_2$. By the above, this is well-defined.

We make another important observation, which we easily derive from the above discussion:
\begin{itemize}
\item Suppose $(C_1,C_2)$ is of type (2), (3a) or (3b). Then there is a unique isomorphism $\mathrm{I}_{C_1,C_2}$ between the 4-cycles $C_1$ and $C_2$, which maps every vertex $x \in C_1$ for which $|\Gamma_1(x) \cap C_2|=1$ to the unique vertex in $\Gamma_1(x) \cap C_2$, every vertex $y \in C_1$ for which $|\Gamma_1(y) \cap C_2|=2$ to one of the two vertices in $\Gamma_1(y) \cap C_2$, and every vertex $z \in C_1$ for which $|\Gamma_1(z) \cap C_2|=3$ to the unique vertex in $\Gamma_1(z) \cap C_2$ that is adjacent to the other two. 
\item Suppose $(C_1,C_2)$ is of type (1). Then we denote by $\mathrm{I}_{C_1,C_2}$ the map which sends each vertex $x$ of $C_1$ to the edge $\Gamma_1(x) \cap C_2$ of $C_2$ and each edge $yz$ of $C_1$ to the unique vertex in $C_2$ adjacent to $y$ and $z$. Then $\mathrm{I}_{C_1,C_2}$ induces an isomorphism between the 4-cycle $C_1$ and the line graph of the 4-cycle $C_2$ and an isomorphism between the line graph of the 4-cycle $C_1$ and the 4-cycle $C_2$. 
\end{itemize}
A straightforward argument shows that $\mathrm{I}_{C_2,C_1} = \mathrm{I}_{C_1,C_2}^{-1}$. It is now clear that if $\mathcal{C} = \{ C_1,C_2,C_3,C_4 \}$ with $\Omega = \{ C_1 \cup C_3,C_2 \cup C_4 \}$, then $\Delta_{1234}:=\mathrm{I}_{C_1,C_2} \mathrm{I}_{C_2,C_3} \mathrm{I}_{C_3,C_4} \mathrm{I}_{C_4,C_1}$ is an automorphism of the 4-cycle $C_1$. Note that the following holds:
\begin{enumerate}
\item[(i)] $\Delta_{1432} = \mathrm{I}_{C_1C_4} \mathrm{I}_{C_4 C_3} \mathrm{I}_{C_3 C_2} \mathrm{I}_{C_2 C_1} = \Delta_{1234}^{-1}$;
\item[(ii)] $\Delta_{2341} = \mathrm{I}_{C_2C_3} \mathrm{I}_{C_3 C_4} \mathrm{I}_{C_4 C_1} \mathrm{I}_{C_1 C_2} = \mathrm{I}_{C_1 C_2}^{-1} \Delta_{1234} \mathrm{I}_{C_1 C_2}$.
\end{enumerate}
For a given type $t$ equal to (1), (2), (3a) or (3b), we now show how to construct up to graph isomorphisms all triples $(\Gamma,\omega_1,\omega_2)$ for which the above conditions (A) and (B) hold. We first describe the method, and then work out each type explicitly. For given ($t$), there is a unique way, up to a relabeling of the vertices in $C_2$, to make the connections between the vertices of $C_1$ and $C_2$. Given these connections, there is then a unique way, up to a relabeling of the vertices in $C_3$, to make the connections between $C_2$ and $C_3$. Giving these extra connections, there is a unique way, up to a relabeling of the vertices in $C_4$, to make the connections between $C_3$ and $C_4$. Finally given all these connections and the value  of $\varphi := \Delta_{1234}$, there will be at most one way to make the final connections between vertices of $C_4$ and $C_1$. The final step (see later) is not always possible for each $\varphi \in D_8$ (if the resulting graph cannot be edge-regular with parameters $(16,6,2)$), but if it is possible then it is only possible in a unique way and we denote the resulting graph by $\Gamma_{t,\varphi}$. Another issue that we need to address is that distinct $\varphi$'s from $D_8$ can result in isomorphic triples $(\Gamma_{t,\varphi},\omega_1,\omega_2)$.

Now, note that $D_8$ has 5 conjugacy classes. Regarding them as sets of automorphisms of the cycle graph with vertices $1 \sim 2 \sim 3 \sim 4 \sim 1$, they are the following:
\[ \{ () \},\{ (1234),(1432) \}, \{ (13)(24) \},\{ (13),(24) \},\{ (12)(34),(14)(23) \}. \]
Note that each of those conjugacy classes is closed under taking inverses. By (i) and (ii) above, we thus see that if $\varphi_1$ and $\varphi_2$ belong to distinct conjugacy classes, then $(\Gamma_{t,\varphi_1},\omega_1,\omega_2)$ and $(\Gamma_{t,\varphi_2},\omega_1,\omega_2)$ can never be isomorphic. By (i), we also know the following. If $\Gamma_{t,\varphi}$ is an edge-regular graph with parameters $(16,6,2)$ for some $\varphi \in \{ (1234),(1432) \}$, then it is also an edge-regular graph for the other element in $\{ (1234),(1432) \}$ and moreover, both associated triples $(\Gamma_{t,\varphi},\omega_1,\omega_2)$ are isomorphic.

\medskip Suppose first the case of type (1). We can relabel the vertices in $C_1$, $C_2$, $C_3$ and $C_4$ such that $\Gamma_{v_1} \cap C_2 = \{ v_9,v_{10} \}$, $\Gamma_{v_2} \cap C_2 = \{ v_{10},v_{11} \}$, $\Gamma_{v_3} \cap C_2 = \{ v_{11},v_{12} \}$, $\Gamma_{v_4} \cap C_2 =\{ v_9,v_{12} \}$, $\Gamma_{v_9} \cap C_3 = \{ v_5,v_6 \}$, $\Gamma_{v_{10}} \cap C_3 = \{ v_6,v_7 \}$, $\Gamma_{v_{11}} \cap C_3 = \{ v_7,v_8 \}$, $\Gamma_{v_{12}} \cap C_3 = \{ v_5,v_8 \}$, $\Gamma_{v_5} \cap C_4 = \{ v_{13},v_{14} \}$, $\Gamma_{v_6} \cap C_4 = \{ v_{14},v_{15} \}$, $\Gamma_{v_7} \cap C_4 = \{ v_{15},v_{16} \}$ and $\Gamma_{v_8} \cap C_4 = \{ v_{13},v_{16} \}$. If we put $\varphi:= \Delta_{1234} \in D_8$, then we must have that $\Gamma_{v_{13}} \cap C_1 = \{ v_{3^\varphi},v_{4^\varphi} \}$, $\Gamma_{v_{14}} \cap C_1 = \{ v_{1^\varphi},v_{4^\varphi} \}$, $\Gamma_{v_{15}} \cap C_1 = \{ v_{1^\varphi},v_{2^\varphi} \}$ and $\Gamma_{v_{16}} \cap C_1 = \{ v_{2^\varphi},v_{3^\varphi} \}$. For each $\varphi \in D_8$, all vertices in $C_1$ have degree $6$; and we have already made sure by construction that all edges are in exactly two triangles. We have already remarked earlier that $(1234)$ and $(1432)$ give rise to isomorphic examples. By symmetry, also $(13)$ and $(24)$ give rise to isomorphic examples, as well as $(12)(34)$ and $(14)(23)$. This case then contributes five triples $(\Gamma,\omega_1,\omega_2)$, namely those where $\Gamma \in \{ \Gamma_{1,()},\Gamma_{1,(1234)},\Gamma_{1,(13)(24)},\Gamma_{1,(13)},\Gamma_{1,(12)(34)} \}$. 

\medskip Suppose next the case of type (2). After relabeling the vertices, we may in this case assume that $\Gamma_{v_1} \cap C_2 = \{ v_9,v_{10},v_{11} \}$, $\Gamma_{v_2} \cap C_2 = \{ v_{11} \}$, $\Gamma_{v_3} \cap C_2 = \{ v_9,v_{11},v_{12} \}$, $\Gamma_{v_4} \cap C_2 = \{ v_9 \}$, $\Gamma_{v_9} \cap C_3 = \{ v_5 \}$, $\Gamma_{v_{10}} \cap C_3 = \{ v_5,v_6,v_7 \}$, $\Gamma_{v_{11}} \cap C_3 = \{ v_7 \}$, $\Gamma_{v_{12}} \cap C_3 = \{ v_5,v_7,v_8 \}$, $\Gamma_{v_5} \cap C_4 = \{ v_{13} \}$, $\Gamma_{v_6} \cap C_4 = \{ v_{13},v_{14},v_{15} \}$, $\Gamma_{v_7} \cap C_4 = \{ v_{15} \}$ and $\Gamma_{v_8} \cap C_4 = \{ v_{13},v_{15},v_{16} \}$. If we put $\varphi:= \Delta_{1234} \in D_8$, then we must have that  $\Gamma_{v_{13}} \cap C_1 = \{ v_{4^\varphi} \}$, $\Gamma_{v_{14}} \cap C_1 = \{ v_{1^\varphi},v_{2^\varphi},v_{4^\varphi} \}$, $\Gamma_{v_{15}} \cap C_1 = \{ v_{2^\varphi} \}$ and $\Gamma_{v_{16}} \cap C_1 = \{ v_{2^\varphi},v_{3^\varphi},v_{4^\varphi} \}$. For the vertices in $C_1$ to have degree 6, we must have that $\{ 1^\varphi,3^\varphi \} = \{ 1,3 \}$ and $\{ 2^\varphi,4^\varphi \} = \{ 2,4 \}$. Note that we have already made sure that all edges are in exactly two triangles. There are thus four possibilities for $\varphi$, namely $()$, $(13)$, $(24)$ and $(13)(24)$. The permutations $()$, $(13)$ and $(13)(24)$ give rise to mutually nonisomorphic triples $(\Gamma,\omega_1,\omega_2)$. We now show that $(13)$ and $(24)$ give rise to isomorphic triples $(\Gamma,\omega_1,\omega_2)$. The distinction between the two cases $(13)$ and $(24)$ is as follows: if $\varphi=(13)$, then $\Delta_{1234}$ interchanges the two vertices (namely 1 and 3) that are adjacent with exactly three vertices of $C_2$, while if $\varphi = (24)$, then $\Delta_{1234}$ interchanges the two vertices (namely 2 and 4) that are adjacent with exactly one vertex of $C_2$. However, if $\varphi = (24)$, then $\Delta_{1432} = \Delta_{1234}^{-1}=(24)$ interchanges the two vertices of $C_1$ that are adjacent to three vertices of $C_4$, yielding a similar situation as for the permutation $(13)$ but with the roles of $C_2$ and $C_4$ interchanged. So, there is a graph isomorphism between $\Gamma_{2,(13)}$ and $\Gamma_{2,(24)}$ fixing each of $C_1$, $C_3$ (and thus also $\omega_1$ and $\omega_2$), but interchanging $C_2$ and $C_4$.

We conclude that this case contributes exactly three triples $(\Gamma,\omega_1,\omega_2)$, namely those where $\Gamma \in \{ \Gamma_{2,()},,\Gamma_{2,(13)(24)},\Gamma_{2,(13)} \}$.     

\medskip Suppose next the case of type (3a). Then we may assume that $\Gamma_{v_1} \cap C_2 = \{ v_9,v_{10},v_{11} \}$, $\Gamma_{v_2} \cap C_2 = \{ v_{11},v_{12} \}$, $\Gamma_{v_3} \cap C_2 = \{ v_9,v_{12} \}$, $\Gamma_{v_4} \cap C_2 = \{ v_9 \}$, $\Gamma_{v_9} \cap C_3 = \{ v_5 \}$, $\Gamma_{v_{10}} \cap C_3 = \{ v_5,v_6,v_7 \}$, $\Gamma_{v_{11}} \cap C_3 = \{ v_7,v_8 \}$, $\Gamma_{v_{12}} \cap C_3 = \{ v_5,v_8 \}$, $\Gamma_{v_5} \cap C_4 = \{ v_{13} \}$, $\Gamma_{v_6} \cap C_4 = \{ v_{13},v_{14},v_{15} \}$, $\Gamma_{v_7} \cap C_4 = \{ v_{15},v_{16} \}$ and $\Gamma_{v_8} \cap C_4 = \{ v_{13},v_{16} \}$. If we put $\varphi := \Delta_{1234} \in D_8$, then we must have that $\Gamma_{v_{13}} \cap C_1 = \{ v_{4^\varphi} \}$, $\Gamma_{v_{14}} \cap C_1 = \{ v_{1^\varphi},v_{2^\varphi},v_{4^\varphi} \}$, $\Gamma_{v_{15}} \cap X_1 = \{ v_{2^\varphi},v_{3^\varphi} \}$ and $\Gamma_{v_{16}} \cap C_1 = \{ v_{3^\varphi},v_{4^\varphi} \}$. For all vertices in $C_1$ to have degree 6, we must have that $\varphi = ()$. Note that we have already made sure that all triangles are in exactly two triangles. So, this case only contributes one triple $(\Gamma_{3a,()},\omega_1,\omega_2)$.

\medskip Suppose next the case of type (3b). Then we may assume that $\Gamma_{v_1} \cap C_2 = \{ v_9,v_{10},v_{11} \}$, $\Gamma_{v_2} \cap C_2 = \{ v_{11},v_{12} \}$, $\Gamma_{v_3} \cap C_2 = \{ v_{12} \}$, $\Gamma_{v_4} \cap C_2 = \{ v_9,v_{12} \}$, $\Gamma_{v_9} \cap C_3 = \{ v_5,v_8 \}$, $\Gamma_{v_{10}} \cap C_3 = \{ v_5,v_6,v_7 \}$, $\Gamma_{v_{11}} \cap C_3 =\{ v_7,v_8 \}$, $\Gamma_{v_{12}} \cap C_3 = \{ v_8 \}$, $\Gamma_{v_5} \cap C_4 = \{ v_{13},v_{16} \}$, $\Gamma_{v_6} \cap C_4 = \{ v_{13},v_{14},v_{15} \}$, $\Gamma_{v_7} \cap C_4 = \{ v_{15},v_{16} \}$ and $\Gamma_{v_8} \cap C_4 = \{ v_{16} \}$. If we put $\varphi:=\Delta_{1234} \in D_8$, then we must have $\Gamma_{v_{13}} \cap C_1 = \{ v_{3^\varphi},v_{4^\varphi} \}$, $\Gamma_{v_{14}} \cap C_1 = \{ v_{1^\varphi},v_{2^\varphi},v_{4^\varphi} \}$, $\Gamma_{v_{15}} \cap C_1 =\{ v_{2^\varphi},v_{3^\varphi} \}$ and $\Gamma_{v_{16}} \cap C_1 = \{ v_{3^\varphi} \}$. For all vertices in $C_1$ to have degree 6, we need $\{ 1^\varphi,3^\varphi \} = \{ 1,3 \}$ and so $\varphi \in \{ (),(24) \}$. Note that we have already made sure that all triangles are in exactly two triangles. So, this case contributes two triples $(\Gamma,\omega_1,\omega_2)$, namely those where $\Gamma \in \{ \Gamma_{3b,()},\Gamma_{3b,(24)} \}$. 

\medskip We have also computationally verified that there are in total 11 nonisomorphic triples $(\Gamma,\omega_1,\omega_2)$ satisfying (A) and (B) above, with $\Gamma$ not a $4 \times 4$-rook graph, by considering all 32 possible combinations of the four types 1, 2, 3a, 3b, and the 8 permutations $\varphi \in D_8$. Our computations hereby confirmed the nonexistences and isomorphisms derived above.

\begin{table}
\begin{center}
\begin{tabular}{|c|c||c|c||c|c|}
\hline
Type & Description & Type & Description & Type & Description \\
\hline  
\hline
Type 1 & $4 \times 4$ rook graph & Type 5 & $\Gamma_{1,(13)}$ & Type 9 & $\Gamma_{2,(13)}$ \\
\hline
Type 2 & $\Gamma_{1,()}$ & Type 6 & $\Gamma_{1,(12)(34)}$ & Type 10 & $\Gamma_{3a,()}$ \\
\hline
Type 3 & $\Gamma_{1,(1234)}$ & Type 7 & $\Gamma_{2,()}$ & Type 11 & $\Gamma_{3b,()}$ \\
\hline
Type 4 & $\Gamma_{1,(13)(24)}$ & Type 8 & $\Gamma_{2,(13)(24)}$ & Type 12 & $\Gamma_{3b,(24)}$ \\
            & Shrikhande graph & & & &  \\
\hline
\end{tabular}
\end{center}
\caption{Classification of the triples $(\Gamma,\omega_1,\omega_2)$ satisfying (A) and (B)} \label{tab1}
\end{table}

\begin{proposition}
The Shrikhande graph has up to isomorphisms one ordered pair $(\omega_1,\omega_2)$ for which property (B) holds. This situation corresponds to Type 4 in Table \ref{tab1}. 
\end{proposition}
\begin{proof}
Consider a partition $\{ \omega_1,\omega_2 \}$ of the $4 \times 4$ rook graph $\Gamma_1$ for which property (B) holds, and let $I$ be an independent set of size 4 contained in $\omega_1$. Consider the Shrikhande graph as the graph $\Gamma_2$ obtained from $\Gamma_1$ by applying so-called (Seidel) switching with respect to $I$, see Example 4.6 of \cite{Cam-vL}. Then $\{ \omega_1,\omega_2 \}$ is still a partition of the vertex set of $\Gamma_2$ for which property (B) holds (but with other adjacencies in $\omega_1$). We show that this situation needs to correspond to type 4, as in all other cases for which the associated graph is not a $4 \times 4$ rook graph (type 1), there exists a independent set of size 4 that is not 2-regular, or the associated graph is not strongly regular with $\mu$-parameter equal to $2$.

For types 7 till 12 (cases 2, (3a) and (3b)), this follows from the fact that the independent set $\{ v_1,v_3,v_5,v_7\}$ is not 2-regular as $\Gamma(v_{10}) \cap \{ v_1,v_3,v_5,v_7 \} = \{ v_1,v_5,v_7 \}$. 

Consider therefore case 1 and the graph $\Gamma_{1,\varphi}$ with $\varphi \in D_8$. The vertices $v_2$ and $v_7$ are nonadjacent and already have $v_{10}$ and $v_{11}$ as common neighbours. The neighbours $v_{15}$ and $v_{16}$ of $v_7$ are therefore not adjacent to $v_2$, which implies that $4^\varphi =2$. Similarly, the vertices $v_3$ and $v_8$ are nonadjacent and already have $v_{11}$ and $v_{12}$ as common neighbours. The neighbours $v_{13}$ and $v_{16}$ of $v_8$ are therefore not adjacent to $v_3$, implying that $1^\varphi = 3$. \

We thus conclude that $\varphi=(13)(24)$ and that type 4 occurs.
\end{proof}

\medskip There are thus up to graph isomorphisms 12 triples $(\Gamma,\omega_1,\omega_2)$, where $\Gamma$ is an edge-regular graph with parameters $(v,k,\lambda)=(16,6,2)$ and $\{ \omega_1,\omega_2 \}$ a partition of the vertex set $V$ of $\Gamma$ in two sets of size $8$ such that the induced subgraph on each $\omega_i$, $i \in \{ 1,2 \}$, is the disjoint union of two $C_4$'s. These 12 types are mentioned in Table \ref{tab1}. Note that each triple $(\Gamma,\omega_1,\omega_2)$ admits four pairs $(\mathcal{P},\mathcal{P}')$, where $\mathcal{P}$ and $\mathcal{P}'$ are two partitions of $V$ in four independent sets of size 4 compatible with $(\omega_1,\omega_2)$.

By Theorem \ref{theo2.1}, we obtain Neumaier graphs by combining or ``stacking'' three such triples, among which certainly one needs to be associated with a $4 \times 4$ rook graph (type 1) in order to obtain a clique of size $4$, necessarily to have a Neumaier graph. If the Neumaier graph is obtained by combining the triple of type 1, with one of type A  and another one of type B, where $1 \leq A \leq B \leq 12$, then we say that the Neumaier graph has {\em type $(1,A,B)$}. There are thus 78 possible types for the constructed Neumaier graphs, resulting in 78 (not necessarily disjoint) families $\mathcal{F}_1,\mathcal{F}_2,\ldots,\mathcal{F}_{78}$ defined by the types $(1,A,B)$, ordered lexicographically (see Tables \ref{tab2} and \ref{tab3}).

Now, fix one of the 78 possible types $(1,A,B)$ and take three triples $(\Gamma_i,\omega_1^{(i)},\omega_2^{(i)})$ of respective types 1, $A$ and $B$. The ordered pair consisting of two partitions of the vertex set of $\Gamma_i$ in four independent sets of size 4 compatible with $(\omega_1^{(i)},\omega_2^{(i)})$ will be denoted by $(\mathcal{P}_i,\mathcal{P}_i')$. As mentioned before, there are for every $i \in \{ 1,2,3 \}$ four possibilities for such a pair. Once we have fixed the three pairs $(\mathcal{P}_1,\mathcal{P}_1')$, $(\mathcal{P}_2,\mathcal{P}_2')$ and $(\mathcal{P}_3,\mathcal{P}_3')$ ($4^3$ possibilities), there are also four possibilities for each of the maps $\varphi_{12}$, $\varphi_{23}$, $\varphi_{31}$, resulting all together in $4^6=4096$ Neumaier graphs by Section \ref{sec2}.

For each of the 78 possible types, we have implemented these 4096 graphs using the computer algebra systems GAP and SageMath, verified computationally that we indeed have Neumaier graphs (as Theorem \ref{theo2.1} says) and determined how many nonisomorphic ones there are. We also determined the characteristic polynomials for the graphs. The conclusion was as follows (see Tables \ref{tab2} and \ref{tab3}):
\begin{itemize}
\item for each of the families $\mathcal{F}_1,\mathcal{F}_2,\ldots,\mathcal{F}_{57}$, all graphs in the family have the same characteristic polynomials;
\item for each of the remaining families $\mathcal{F}_{58},\mathcal{F}_{59},\ldots,\mathcal{F}_{78}$, there are three characteristic polynomials in the family. Such a family $\mathcal{F}_i$ can thus be subdivided into three subfamilies $\mathcal{F}_{iA}$, $\mathcal{F}_{iB}$, $\mathcal{F}_{iC}$. In fact, this subdivision, can be made such that $(|\mathcal{F}_{iA}|,|\mathcal{F}_{iB}|,|\mathcal{F}_{iC}|)$ is equal to either $(2,4,10)$, $(2,6,14)$, $(2,8,20)$ or $(4,6,14)$.
\end{itemize}
We have also computationally verified that any two of the 78 families either completely overlap (when the graphs in one family can be paired with the graphs in the other family, by isomorphism), or have no overlap at all (if no isomorphisms exist between graphs of the families). In fact, we have found an heuristic rule for determining whether two families with respective types $(1,A_1,B_1)$ and $(1,A_2,B_2)$, with $(A_1,B_1) \not= (A_2,B_2)$ completely overlap. For this to be the case, we need that $(A_2,B_2)$ is obtained from $(A_1,B_1)$ by replacing all $5$'s by $6$'s, and all $6$'s by $5$'s. According to this heuristic rule, for instance, the families of types $(1,5,5)$ and $(1,6,6)$ coincide, but are still distinct from the family of type $(1,5,6)$ (although the graphs still have the same spectrum, see Table \ref{tab2}). 

According to what just has been said, we can now consider the 109 families of mutually disjoint Neumaier graphs. These can be found in Tables \ref{tab2} and \ref{tab3}, along with the number of nonisomorphic graphs contained in it, and the number of eigenvalues of their graphs. We note that there are 1063 graphs in total, and that the number of eigenvalues is either 5 or a number in the interval $[8,20]$. We also notice that there are exactly 25 Neumaier graphs in the table with precisely 5 eigenvalues and that all these graphs are obtained by combining $4 \times 4$ rook graphs and Shrikhande graphs (with at least one rook graph). This answers an open problem from the literature as to whether Neumaier graphs with exactly give eigenvalues can exist. In the next section, we prove in a theoretical way that all Neumaier graphs that arise from rook graphs and Shrikhande graphs indeed have five eigenvalues.

We have also computed the spectrum of all graphs. We note that 14 (the degree), 6, 2, -2 and -6 are always eigenvalues and that the multiplicities of $14$ and $6$ are always 1. There are also 20 polynomials that can occur in the factorization of the characteristic polynomial as a product of irreducible factors over $\mathbb{Q}$. These polynomials are $p_{1A}(x)=x-14$, $p_{1B}(x)=x-6$, $p_{1C}(x)=x-2$, $p_{1D}(x)=x+2$, $p_{1E}(x)=x+6$, $p_{1F}(x)=x$, $P_{1G}(x)=x-4$, and the thirteen polynomials that can be found  in Table \ref{tab4} along with their roots. The spectrum of the 1063 graphs has also been mentioned in Tables \ref{tab2} and \ref{tab3}. These spectra are composed as follows:
\begin{itemize}
\item an eigenvalue $\lambda \in \{ 14,6,2,-2,-6,0,4,\sqrt{2},-\sqrt{2},\sqrt{8},-\sqrt{8},-2+2\sqrt{5},-2-2\sqrt{5},1+\sqrt{5},-1-\sqrt{5} \}$ with multiplicity $m$ yields a contribution to the spectrum, denoted by $\lambda^m$;
\item the roots of any polynomial $p_X(x)$ with $X \in \{ 3A,3B,3C,4A,4B,4C,4D,5A,5B \}$ as listed in Table \ref{tab4} that occurs $M$ times in the factorization of the characteristic polynomial as a product of irreducible factors over $\mathbb{Q}$ yields a contribution to the spectrum, denoted by $(X)^m$.   
\end{itemize}
From Tables \ref{tab2} and \ref{tab3}, we see that some of the $109$ families have graphs with the same spectra (i.e. characteristic polynomials). Again, we can provide an heuristic rule as to when this is precisely the case:
\begin{itemize}
\item if $(N_1,N_2) \in \{ (59,69),(63,71),(67,72)\}$ and $X \in \{ A,B,C \}$, then the graphs in the families $N_1X$ and $N_2X$ have the same spectra.
\item Families of types $(1,A_1,B_1)$ and $(1,A_2,B_2)$ give rise to the same spectrum if the multiset $\{ A_2,B_2 \}$ can be obtained from the multiset $\{ A_1,B_1 \}$ by interchanging some of the $1$'s by $4$'s.
\item As mentioned before, the families of types (1,5,5) (or (1,6,6)) and (1,5,6) have graphs with the same spectrum. This spectrum is moreover the same as the spectra of the graphs in the families $(1,1,2)$ and $(1,2,4)$. 
\end{itemize}
Taking these heuristic rules into account, 87 distinct spectra remain for the 1063 Neumaier graphs (see Tables \ref{tab2} and \ref{tab3}).

\section{Determination of the spectrum for the graphs in families 1, 4 and 34} \label{sec5}

Consider again the construction of Section \ref{sec2}, and suppose that each of $\Gamma_1=(V_1,E_1)$, $\Gamma_2=(V_2,E_2)$, $\Gamma_3=(V_3,E_3)$ is a $4 \times 4$ rook graph or a Shrikhande graph. Let $\overline{\Gamma}$ be a graph as obtained in Section \ref{sec2} after making particular choices for $\Omega_1$, $\Omega_2$, $\Omega_3$, $(\mathcal{P}_1,\mathcal{P}_1')$, $(\mathcal{P}_2,\mathcal{P}_2')$, $(\mathcal{P}_3,\mathcal{P}_3')$, $\varphi_{12}$, $\varphi_{23}$ and $\varphi_{31}$. Note that we do not assume that one of the $\Gamma_i$'s is a rook graph. So, $\overline{\Gamma}$ is not necessarily a Neumaier graph.

The vertex set $V$ of $\overline{\Gamma}$ is then the disjoint union of $V_1$, $V_2$ and $V_3$. We now partition the set $V \times V$ in nine subsets $R_i$, $i \in \{ 1,2,\ldots,9 \}$.

Consider two vertices $x \in V_i$ and $y \in V_j$ with $i,j \in \{ 1,2,3 \}$. Let $\omega_x$ denote the element of $\Omega_i$ containing $x$, and let $I_x \subset \omega_x$ denote the intersection of the elements of $\mathcal{P}_i$ and $\mathcal{P}_i'$ that contain $x$. 

\medskip \noindent Suppose first that $i=j$. Then we say that 

$\bullet$ $(x,y) \in R_1$ if $y=x$, 

$\bullet$ $(x,y) \in R_2$ if $y \in \omega_x$ and $y \sim x$,

$\bullet$ $(x,y) \in R_3$ if $y \not= x$ and $y \in I_x$,

$\bullet$ $(x,y) \in R_4$ if $y \in \omega_x$, $y \not=x$, $y \not\sim x$ and $y \notin I_x$, 

$\bullet$ $(x,y) \in R_5$ if $y \not\in \omega_x$ and $y \sim x$,

$\bullet$ $(x,y) \in R_6$ if $y \not\in \omega_x$ and $y \not\sim x$.

\noindent If $i \not= j$, then we say that

$\bullet$ $(x,y) \in R_7$ if $y \sim x$;

$\bullet$ $(x,y) \in R_8$ if $y \in \omega_x^{\varphi_{ij}}$ and $y \not\sim x$;

$\bullet$ $(x,y) \in R_9$ if $y \not\in \omega_x^{\varphi_{ij}}$.

\medskip \noindent It is clear that $R_1,R_2,\ldots,R_9$ are symmetric relations on the set $V$. Recall that the graphs $\Gamma_1$, $\Gamma_2$, $\Gamma_3$ are strongly regular with parameters $(16,6,2,2)$, and that every independent set of size 4 in any of these graphs is regular with nexus 2. Also, if we define $E_r$, $E_s$, $E_t$ and $E_u$ as in Section \ref{sec4}, then we have $|E_r|=|E_s|=|E_t|=|E_u|=1$ for both the $4 \times 4$ rook graph as the Shrikhande graph. Using the facts, it is then straightforward to show the following:
\begin{quote}
For $i,j \in \{ 1,2,\ldots,9 \}$ and $(x,y) \in R_i$, let $M_{ij}(x,y)$ denote the number of vertices $z$ adjacent to $y$ for which $(x,z) \in R_j$. Then $M_{ij}(x,y)$ only depends on $i$ and $j$ and not on the particular choice of the vertices $x$ and $y$ for which $(x,y) \in R_i$. Moreover, the numbers $M_{ij}$ are as in the following matrix: 
\end{quote}
\begin{displaymath}
M = \left[
\begin{array}{ccccccccc}
0 & 2 & 0 & 0 & 4 & 0 & 8 & 0 & 0 \\
1 & 0 & 1 & 0 & 2 & 2 & 0 & 8 & 0 \\
0 & 2 & 0 & 0 & 0 & 4 & 8 & 0 & 0 \\
0 & 0 & 0 & 2 & 2 & 2 & 4 & 4 & 0 \\
1 & 1 & 0 & 2 & 1 & 1 & 0 & 0 & 8 \\
0 & 1 & 1 & 2 & 1 & 1 & 0 & 0 & 8 \\
1 & 0 & 1 & 2 & 0 & 0 & 2 & 4 & 4 \\ 
0 & 2 & 0 & 2 & 0 & 0 & 4 & 2 & 4 \\
0 & 0 & 0 & 0 & 2 & 2 & 2 & 2 & 6 
\end{array}
\right].
\end{displaymath}

The matrix $M$ has eigenvalues $\lambda_1=14$, $\lambda_2=6$, $\lambda_3=2$, $\lambda_4=-2$ and $\lambda_5=-6$, with respective multiplicities $1$, $1$, $4$, $1$ and $2$. Following a reasoning completely similar to the one on page 1226 of \cite{BDB:2}, this means that the adjacency matrix $A$ of $\overline{\Gamma}$ also has these eigenvalues. Denote by $m_i$, $i \in \{ 1,2,3,4,5 \}$, the multiplicity of $\lambda_i$ regarded as an eigenvalue of $A$. As $\overline{\Gamma}$ is a connected $14$-regular graph, we have $m_1=1$. Note that $\overline{\Gamma}$ is an edge-regular graph with parameters $(v,k',\lambda)=(48,14,2)$. From $\Sigma m_i = v$, $\Sigma \lambda_i m_i =0$, $\sum \lambda_i^2 m_i = vk'$ and $\sum \lambda_i^3 m_i = v k' \lambda$, it readily follows that $m_2=1$, $m_3=26$, $m_4=12$ and $m_5=8$. So, $\overline{\Gamma}$ is a graph with spectrum $14^1 6^1 2^{26} (-2)^{12} (-6)^8$.

\medskip \noindent \textbf{Remark.} This case includes besides $25$ Neumaier graphs (Families 1, 4 and 34) also 9 mutually nonisomorphic edge-regular graphs with parameters $(48,14,2)$ that are non-Neumaier, and which are obtained by stacking three Shrikhande graphs. 

\begin{center}
\begin{table}
{\scriptsize
\begin{tabular}{|c||c|c|c|c|}
\hline
Family & \# graphs & \# eigenvalues & Spectrum & Type \\
\hline
\hline
1 & 4 & 5  & $14^1 6^1 2^{26} (-2)^{12} (-6)^8 $ & (1,1,1) \\
\hline
2 & 7 & 8 & $14^1 6^1 2^{22} (-2)^{8} (-6)^{8} 0^4 (\pm \sqrt{8})^2$ & (1,1,2) \\
\hline
3 & 3 & 9 & $14^1 6^1 2^{22} (-2)^{8} (-6)^{8} (4A)^2$ & (1,1,3) \\
\hline
4 & 7 & 5 & $14^1 6^1 2^{26} (-2)^{12} (-6)^{8}$ & (1,1,4) \\
\hline
5/6 & 4 & 8 & $14^1 6^1 2^{24} (-2)^{10} (-6)^{8} 0^2 (\pm \sqrt{8})^1$ & (1,1,5)/(1,1,6) \\
\hline
7 & 7 & 9 & $14^1 6^1 2^{20} (-2)^{10} (-6)^{6} 0^4 (3B)^2$ & (1,1,7) \\
\hline
8 & 7 & 10 & $14^1 6^1 2^{18} (-2)^{8} (-6)^{6} (\pm \sqrt{2})^4 (3B)^2$ & (1,1,8) \\
\hline
9 & 7 & 11 & $14^1 6^1 2^{19} (-2)^{9} (-6)^{6} 0^2 (\pm \sqrt{2})^2 (3B)^2$ & (1,1,9) \\
\hline
10 & 7 & 9 & $14^1 6^1 2^{22} (-2)^{10} (-6)^{6} (4C)^2$ & (1,1,10) \\
\hline
11 & 7 & 9 & $14^1 6^1 2^{22} (-2)^{10} (-6)^{6} 0^2 (3A)^2$ & (1,1,11) \\
\hline
12 & 7 & 10 & $14^1 6^1 2^{21} (-2)^{9} (-6)^{6} (\pm \sqrt{2})^2 (3A)^2$ & (1,1,12) \\
\hline
13 & 14 & 8 & $14^1 6^1 2^{18} (-2)^{4} (-6)^{8} 0^8 (\pm \sqrt{8})^4$ & (1,2,2) \\
\hline
14 & 6 & 12 & $14^1 6^1 2^{18} (-2)^{4} (-6)^{8} 0^4 (\pm \sqrt{8})^2 (4A)^2$ & (1,2,3) \\
\hline
15 & 20 & 8 & $14^1 6^1 2^{22} (-2)^{8} (-6)^{8} 0^4 (\pm \sqrt{8})^2$ & (1,2,4) \\
\hline
16/17 & 16 & 8 & $14^1 6^1 2^{20} (-2)^{6} (-6)^{8} 0^6 (\pm \sqrt{8})^3$ & (1,2,5)/(1,2,6) \\
\hline
18 & 20 & 11 & $14^1 6^1 2^{16} (-2)^{6} (-6)^{6} 0^8 (\pm \sqrt{8})^2 (3B)^2$ & (1,2,7) \\
\hline
19 & 20 & 13 & $14^1 6^1 2^{14} (-2)^{4} (-6)^{6} 0^4 (\pm \sqrt{2})^4 (\pm \sqrt{8})^2 (3B)^2$ & (1,2,8) \\
\hline
20 & 20 & 13 & $14^1 6^1 2^{15} (-2)^{5} (-6)^{6} 0^6 (\pm \sqrt{2})^2 (\pm \sqrt{8})^2 (3B)^2$ & (1,2,9) \\
\hline
21 & 20 & 12 & $14^1 6^1 2^{18} (-2)^{6} (-6)^{6} 0^4 (\pm \sqrt{8})^2 (4C)^2$ & (1,2,10) \\
\hline
22 & 20 & 11 & $14^1 6^1 2^{18} (-2)^{6} (-6)^{6} 0^6 (\pm \sqrt{8})^2 (3A)^2$ & (1,2,11) \\
\hline
23 & 20 & 13 & $14^1 6^1 2^{17} (-2)^{5} (-6)^{6} 0^4 (\pm \sqrt{2})^2 (\pm \sqrt{8})^2 (3A)^2$ & (1,2,12) \\
\hline
24 & 4 & 9 & $14^1 6^1 2^{18} (-2)^{4} (-6)^{8} (4A)^4$ & (1,3,3) \\
\hline
25 & 6 & 9 & $14^1 6^1 2^{22} (-2)^{8} (-6)^{8} (4A)^2$ & (1,3,4) \\
\hline
26/27 & 4 & 12 & $14^1 6^1 2^{20} (-2)^{6} (-6)^{8} 0^2 (\pm \sqrt{8})^1 (4A)^2$ & (1,3,5)/(1,3,6) \\
\hline
28 & 8 & 13 & $14^1 6^1 2^{16} (-2)^{6} (-6)^{6} 0^4 (3B)^2 (4A)^2$ & (1,3,7) \\
\hline
29 & 8 & 14 & $14^1 6^1 2^{14} (-2)^{4} (-6)^{6} (\pm \sqrt{2})^4 (3B)^2 (4A)^2$ & (1,3,8) \\
\hline
30 & 8 & 15 & $14^1 6^1 2^{15} (-2)^{5} (-6)^{6} 0^2 (\pm \sqrt{2})^2 (3B)^2 (4A)^2$ & (1,3,9) \\
\hline
31 & 8 & 13 & $14^1 6^1 2^{18} (-2)^{6} (-6)^{6} (4A)^2 (4C)^2$ & (1,3,10) \\
\hline
32 & 8 & 13  & $14^1 6^1 2^{18} (-2)^{6} (-6)^{6} 0^2 (3A)^2 (4A)^2$ & (1,3,11) \\
\hline
33 & 8 & 14  & $14^1 6^1 2^{17} (-2)^{5} (-6)^{6} (\pm \sqrt{2})^2 (3A)^2 (4A)^2$ & (1,3,12) \\
\hline
34 & 14 & 5 & $14^1 6^1 2^{26} (-2)^{12} (-6)^{8}$ & (1,4,4) \\
\hline
35/36 & 16 & 8 & $14^1 6^1 2^{24} (-2)^{10} (-6)^{8} 0^2 (\pm \sqrt{8})^1$ & (1,4,5)/(1,4,6) \\
\hline
37 & 20 & 9 & $14^1 6^1 2^{20} (-2)^{10} (-6)^{6} 0^4 (3B)^2$ & (1,4,7) \\
\hline
38 & 20 & 10 & $14^1 6^1 2^{18} (-2)^{8} (-6)^{6} (\pm \sqrt{2})^4 (3B)^2$ & (1,4,8) \\
\hline
39 & 20 & 11 & $14^1 6^1 2^{19} (-2)^{9} (-6)^{6} 0^2 (\pm \sqrt{2})^2 (3B)^2$ & (1,4,9) \\
\hline
40 & 20 & 9 & $14^1 6^1 2^{22} (-2)^{10} (-6)^{6} (4C)^2$ & (1,4,10) \\
\hline
41 & 20 & 9 & $14^1 6^1 2^{22} (-2)^{10} (-6)^{6} 0^2 (3A)^2$ & (1,4,11) \\
\hline
42 & 20 & 10 & $14^1 6^1 2^{21} (-2)^{9} (-6)^{6} (\pm \sqrt{2})^2 (3A)^2$ & (1,4,12) \\
\hline
43/51 & 12 & 8 & $14^1 6^1 2^{22} (-2)^{8} (-6)^{8} 0^4 (\pm \sqrt{8})^2$ & (1,5,5)/(1,6,6) \\
\hline
44 & 6 & 8 & $14^1 6^1 2^{22} (-2)^{8} (-6)^{8} 0^4 (\pm \sqrt{8})^2$ & (1,5,6) \\
\hline
45/52 & 16 & 11 & $14^1 6^1 2^{18} (-2)^{8} (-6)^{6} 0^6 (\pm \sqrt{8})^1 (3B)^2$ & (1,5,7)/(1,6,7) \\
\hline
46/53 & 16 & 13 & $14^1 6^1 2^{16} (-2)^{6} (-6)^{6} 0^2 (\pm \sqrt{2})^4 (\pm \sqrt{8})^1 (3B)^2$ & (1,5,8)/(1,6,8) \\
\hline
47/54 & 24 & 13 & $14^1 6^1 2^{17} (-2)^{7} (-6)^{6} 0^4 (\pm \sqrt{2})^2 (\pm \sqrt{8})^1 (3B)^2$ & (1,5,9)/(1,6,9) \\
\hline
48/55 & 16 & 12 & $14^1 6^1 2^{20} (-2)^{8} (-6)^{6} 0^2 (\pm \sqrt{8})^1 (4C)^2$ & (1,5,10)/(1,6,10) \\
\hline
49/56 & 24 & 11 & $14^1 6^1 2^{20} (-2)^{8} (-6)^{6} 0^4 (\pm \sqrt{8})^1 (3A)^2$ & (1,5,11)/(1,6,11) \\
\hline
50/57 & 24 & 13 & $14^1 6^1 2^{19} (-2)^{7} (-6)^{6} 0^2 (\pm \sqrt{2})^2 (\pm \sqrt{8})^1 (3A)^2$ & (1,5,12)/(1,6,12) \\
\hline
58A & 2 & 9 & $14^1 6^1 2^{16} (-2)^{8} (-6)^{6} 0^{10} 4^2 (-2 \pm 2 \sqrt{5})^2$ & (1,7,7) \\
\hline
58B & 4 & 12 & $14^1 6^1 2^{15} (-2)^{8} (-6)^{5} 0^9 4^1 (-2 \pm 2 \sqrt{5})^1 (3B)^1$ & (1,7,7) \\
\hline
58C & 10 & 9 & $14^1 6^1 2^{14} (-2)^{8} (-6)^{4} 0^8 (3B)^4$ & (1,7,7) \\
\hline
59A & 2 & 11 & $14^1 6^1 2^{14} (-2)^{6} (-6)^{6} 0^6 4^2 (\pm \sqrt{2})^4 (-2 \pm 2\sqrt{5})^2$ & (1,7,8) \\
\hline
59B & 6 & 14 & $14^1 6^1 2^{13} (-2)^{6} (-6)^{5} 0^5 4^1 (\pm \sqrt{2})^4 (-2 \pm 2 \sqrt{5})^1 (3B)^2$ & (1,7,8) \\
\hline
59C & 14 & 11 & $14^1 6^1 2^{12} (-2)^{6} (-6)^{4} 0^4 (\pm \sqrt{2})^4 (3B)^4$ & (1,7,8) \\
\hline
60A & 2 & 11 & $14^1 6^1 2^{15} (-2)^{7} (-6)^{6} 0^8 4^2 (\pm \sqrt{2})^2 (-2 \pm 2 \sqrt{5})^2$ & (1,7,9) \\
\hline
60B & 6 & 14 & $14^1 6^1 2^{14} (-2)^{7} (-6)^{5} 0^7 4^1 (\pm \sqrt{2})^2 (-2 \pm 2 \sqrt{5})^1 (3B)^2$ & (1,7,9) \\
\hline
60C & 14 & 11 & $14^1 6^1 2^{13} (-2)^{7} (-6)^{4} 0^6 (\pm \sqrt{2})^2 (3B)^4$ & (1,7,9) \\
\hline
\end{tabular}}
\caption{The 109 families of Neumaier graphs with parameters $(48,14,2;1,4)$, part I} \label{tab2}
\end{table}
\end{center}
 
\begin{center}
\begin{table}
{\scriptsize
\begin{tabular}{|c||c|c|c|c|}
\hline
Family & \# graphs & \# eigenvalues & Spectrum & Type \\
\hline
\hline
61A & 2 & 11  & $14^1 6^1 2^{16} (-2)^{8} (-6)^{6} 0^6 (5B)^2$ & (1,7,10) \\
\hline
61B & 6 & 18 & $14^1 6^1 2^{16} (-2)^{8} (-6)^{5} 0^5 (3B)^1 (4C)^1 (5B)^1$  & (1,7,10) \\
\hline
61C & 14 & 13 & $14^1 6^1 2^{16} (-2)^{8} (-6)^{4} 0^4 (3B)^2 (4C)^2$ & (1,7,10) \\
\hline
62A & 2 & 11 & $14^1 6^1 2^{16} (-2)^{8} (-6)^{6} 0^6 (5A)^2$  & (1,7,11) \\
\hline
62B & 6 & 17 & $14^1 6^1 2^{16} (-2)^{8} (-6)^{5} 0^6 (3A)^1 (3B)^1 (5A)^1$  & (1,7,11) \\
\hline
62C & 14 & 12 & $14^1 6^1 2^{16} (-2)^{8} (-6)^{4} 0^6 (3A)^2 (3B)^2$ & (1,7,11) \\
\hline
63A & 2 & 13 & $14^1 6^1 2^{15} (-2)^{7} (-6)^{6} 0^4 (\pm \sqrt{2})^2 (5A)^2$  & (1,7,12) \\
\hline
63B & 6 & 19 & $14^1 6^1 2^{15} (-2)^{7} (-6)^{5} 0^4 (\pm \sqrt{2})^2 (3A)^1 (3B)^1 (5A)^1$  & (1,7,12) \\
\hline
63C & 14 & 14 & $14^1 6^1 2^{15} (-2)^{7} (-6)^{4} 0^4  (\pm \sqrt{2})^2 (3A)^2 (3B)^2$ & (1,7,12) \\
\hline
64A & 2 & 11 & $14^1 6^1 2^{12} (-2)^{4} (-6)^{6} 0^2 4^2 (\pm \sqrt{2})^8 (-2 \pm 2 \sqrt{5})^2$ & (1,8,8) \\
\hline
64B & 4 & 14 & $14^1 6^1 2^{11} (-2)^{4} (-6)^{5} 0^1 4^1 (\pm \sqrt{2})^8 (-2 \pm 2 \sqrt{5})^1 (3B)^2$ & (1,8,8) \\
\hline
64C & 10 & 10 & $14^1 6^1 2^{10} (-2)^{4} (-6)^{4} (\pm \sqrt{2})^8 (3B)^4$ & (1,8,8) \\
\hline
65A & 2 & 11 & $14^1 6^1 2^{13} (-2)^{5} (-6)^{6} 0^4 4^2 (\pm \sqrt{2})^6 (-2 \pm 2 \sqrt{5})^2$ & (1,8,9) \\
\hline
65B & 6 & 14 & $14^1 6^1 2^{12} (-2)^{5} (-6)^{5} 0^3 4^1 (\pm \sqrt{2})^6 (-2 \pm 2 \sqrt{5})^1 (3B)^2$ & (1,8,9) \\
\hline
65C & 14 & 11 & $14^1 6^1 2^{11} (-2)^{5} (-6)^{4} 0^2 (\pm \sqrt{2})^6 (3B)^4$ & (1,8,9) \\
\hline
66A & 2 & 13 & $14^1 6^1 2^{14} (-2)^{6} (-6)^{6} 0^2 (\pm \sqrt{2})^4 (5B)^2$ & (1,8,10) \\
\hline
66B & 6 & 20 & $14^1 6^1 2^{14} (-2)^{6} (-6)^{5} 0^1 (\pm \sqrt{2})^4 (3B)^1 (4C)^1 (5B)^1$ & (1,8,10) \\
\hline
66C & 14 & 14 & $14^1 6^1 2^{14} (-2)^{6} (-6)^{4} (\pm \sqrt{2})^4 (3B)^2 (4C)^2$ & (1,8,10) \\
\hline
67A & 2 & 13 & $14^1 6^1 2^{14} (-2)^{6} (-6)^{6} 0^2 (\pm \sqrt{2})^4 (5A)^2$ & (1,8,11) \\
\hline
67B & 6 & 19 & $14^1 6^1 2^{14} (-2)^{6} (-6)^{5} 0^2 (\pm \sqrt{2})^4 (3A)^1 (3B)^1 (5A)^1$ & (1,8,11) \\
\hline
67C & 14 & 14 & $14^1 6^1 2^{14} (-2)^{6} (-6)^{4} 0^2 (\pm \sqrt{2})^4 (3A)^2 (3B)^2$ & (1,8,11) \\
\hline
68A & 2 & 12 & $14^1 6^1 2^{13} (-2)^{5} (-6)^{6} (\pm \sqrt{2})^6 (5A)^2$ & (1,8,12) \\
\hline
68B & 6 & 18 & $14^1 6^1 2^{13} (-2)^{5} (-6)^{5} (\pm \sqrt{2})^6 (3A)^1 (3B)^1 (5A)^1$ & (1,8,12) \\
\hline
68C & 14 & 13 & $14^1 6^1 2^{13} (-2)^{5} (-6)^{4} (\pm \sqrt{2})^6 (3A)^2 (3B)^2$ & (1,8,12) \\
\hline
69A & 4 & 11 & $14^1 6^1 2^{14} (-2)^{6} (-6)^{6} 0^6 4^2 (\pm \sqrt{2})^4 (-2 \pm 2 \sqrt{5})^2$ & (1,9,9) \\
\hline
69B & 6 & 14 & $14^1 6^1 2^{13} (-2)^{6} (-6)^{5} 0^5 4^1 (\pm \sqrt{2})^4 (-2 \pm 2 \sqrt{5})^1 (3B)^2$ & (1,9,9) \\
\hline
69C & 14 & 11 & $14^1 6^1 2^{12} (-2)^{6} (-6)^{4} 0^4 (\pm \sqrt{2})^4 (3B)^4$ & (1,9,9) \\
\hline
70A & 2 & 13 & $14^1 6^1 2^{15} (-2)^{7} (-6)^{6} 0^4 (\pm \sqrt{2})^2 (5B)^2$ & (1,9,10) \\
\hline
70B & 6 & 20  & $14^1 6^1 2^{15} (-2)^{7} (-6)^{5} 0^3 (\pm \sqrt{2})^2 (3B)^1 (4C)^1 (5B)^1$ & (1,9,10) \\
\hline
70C & 14 & 15 & $14^1 6^1 2^{15} (-2)^{7} (-6)^{4} 0^2 (\pm \sqrt{2})^2 (3B)^2 (4C)^2$ & (1,9,10) \\
\hline
71A & 2 & 13 & $14^1 6^1 2^{15} (-2)^{7} (-6)^{6} 0^4 (\pm \sqrt{2})^2 (5A)^2$ & (1,9,11) \\
\hline
71B & 8 & 19 & $14^1 6^1 2^{15} (-2)^{7} (-6)^{5} 0^4 (\pm \sqrt{2})^2 (3A)^1 (3B)^1 (5A)^1$ & (1,9,11) \\
\hline
71C & 20 & 14 & $14^1 6^1 2^{15} (-2)^{7} (-6)^{4} 0^4 (\pm \sqrt{2})^2 (3A)^2 (3B)^2$ & (1,9,11) \\
\hline
72A & 2 & 13 & $14^1 6^1 2^{14} (-2)^{6} (-6)^{6} 0^2 (\pm \sqrt{2})^4 (5A)^2$ & (1,9,12) \\
\hline
72B & 8 & 19 & $14^1 6^1 2^{14} (-2)^{6} (-6)^{5} 0^2 (\pm \sqrt{2})^4 (3A)^1 (3B)^1 (5A)^1$ & (1,9,12) \\
\hline
72C & 20 & 14 & $14^1 6^1 2^{14} (-2)^{6} (-6)^{4} 0^2 (\pm \sqrt{2})^4 (3A)^2 (3B)^2$ & (1,9,12) \\
\hline
73A & 2 & 12 & $14^1 6^1 2^{18} (-2)^{8} (-6)^{6} 0^2 (1 \pm \sqrt{5})^2 (4D)^2$ & (1,10,10) \\
\hline
73B & 4 & 16 & $14^1 6^1 2^{18} (-2)^{8} (-6)^{5} 0^1 (1 \pm \sqrt{5})^1 (4C)^2 (4D)^1$ & (1,10,10) \\
\hline
73C & 10 & 9 & $14^1 6^1 2^{18} (-2)^{8} (-6)^{4} (4C)^4$ & (1,10,10) \\
\hline
74A & 2 & 12 & $14^1 6^1 2^{18} (-2)^{8} (-6)^{6} 0^2 (1 \pm \sqrt{5})^2 (4B)^2$ & (1,10,11) \\
\hline
74B & 6 & 19  & $14^1 6^1 2^{18} (-2)^{8} (-6)^{5} 0^2 (1 \pm \sqrt{5})^1 (3A)^1 (4B)^1 (4C)^1$ & (1,10,11) \\
\hline
74C & 14 & 13 & $14^1 6^1 2^{18} (-2)^{8} (-6)^{4} 0^2 (3A)^2 (4C)^2$ & (1,10,11) \\
\hline
75A & 2 & 13 & $14^1 6^1 2^{17} (-2)^{7} (-6)^{6} (\pm \sqrt{2})^2 (1 \pm \sqrt{5})^2 (4B)^2$ & (1,10,12) \\
\hline
75B & 6 & 20 & $14^1 6^1 2^{17} (-2)^{7} (-6)^{5} (\pm \sqrt{2})^2 (1 \pm \sqrt{5})^1 (3A)^1 (4B)^1 (4C)^1$ & (1,10,12) \\
\hline
75C & 14 & 14 & $14^1 6^1 2^{17} (-2)^{7} (-6)^{4} (\pm \sqrt{2})^2 (3A)^2 (4C)^2$ & (1,10,12) \\
\hline
76A & 2 & 11 & $14^1 6^1 2^{18} (-2)^{8} (-6)^{6} 0^4 (1 \pm \sqrt{5})^2 (3C)^2$ & (1,11,11) \\
\hline
76B & 6 & 14 & $14^1 6^1 2^{18} (-2)^{8} (-6)^{5} 0^4 (1 \pm \sqrt{5})^1 (3A)^2 (3C)^1$ & (1,11,11) \\
\hline
76C & 14 & 9 & $14^1 6^1 2^{18} (-2)^{8} (-6)^{4} 0^4 (3A)^4$ & (1,11,11) \\
\hline
77A & 2 & 13 & $14^1 6^1 2^{17} (-2)^{7} (-6)^{6} 0^2 (\pm \sqrt{2})^2 (1 \pm \sqrt{5})^2 (3C)^2$ & (1,11,12) \\
\hline
77B & 8 & 16 & $14^1 6^1 2^{17} (-2)^{7} (-6)^{5} 0^2 (\pm \sqrt{2})^2 (1 \pm \sqrt{5})^1 (3A)^2 (3C)^1$ & (1,11,12) \\
\hline
77C & 20 & 11 & $14^1 6^1 2^{17} (-2)^{7} (-6)^{4} 0^2 (\pm \sqrt{2})^2 (3A)^4$ & (1,11,12) \\
\hline
78A & 2 & 12 & $14^1 6^1 2^{16} (-2)^{6} (-6)^{6} (\pm \sqrt{2})^4 (1 \pm \sqrt{5})^2 (3C)^2$ & (1,12,12) \\
\hline
78B & 6 & 15 & $14^1 6^1 2^{16} (-2)^{6} (-6)^{5} (\pm \sqrt{2})^4 (1 \pm \sqrt{5})^1 (3A)^2 (3C)^1$ & (1,12,12) \\
\hline
78C & 14 & 10 & $14^1 6^1 2^{16} (-2)^{6} (-6)^{4} (\pm \sqrt{2})^4 (3A)^4$ & (1,12,12) \\
\hline
\end{tabular}}
\caption{The 109 families of Neumaier graphs with parameters $(48,14,2;1,4)$, part II} \label{tab3}
\end{table}
\end{center}

\begin{table}
\begin{center}
\begin{tabular}{|c|c|}
\hline
Irreducible polynomial & Roots \\
\hline
\hline
$p_{2A}(x)=x^2-2$ & $-\sqrt{2} \approx -1.414... < \sqrt{2} \approx 1.414...$ \\
\hline
$p_{2B}(x)=x^2-8$ & $-\sqrt{8} \approx -2.828... < \sqrt{8} \approx 2.828...$ \\
\hline
$p_{2C}(x)=x^2-2x-4$ & $1-\sqrt{5} \approx -1.236... <  1 + \sqrt{5} \approx 3.236...$  \\
\hline
$p_{2D}(x)=x^2-4x-16$ & $2-2\sqrt{5} \approx -2.472... <  2 + 2\sqrt{5} \approx 6.472...$ \\
\hline
$p_{3A}(x) = x^3 + 4x^2 -16x -16$ & $ -6.172... < -0.856... <  3.028...$ \\
\hline
$p_{3B}(x) = x^3 + 2x^2 -24 x +16$ & $ -6.249... < 0.726... < 3.523...$ \\
\hline
$p_{3C}(x)=x^3+4x^2-16x-8$ & $ -6.328... <  -0.454... <  2.782...$ \\
\hline
$p_{4A}(x)=x^4-8x^2+8$ & $ -2.613... < -1.082... < 1.082... < 2.613...$  \\
\hline
$p_{4B}(x)=x^4+4x^3-16x^2-8x+8$ & $ -6.304... < -0.899... < 0.527... < 2.676...$ \\
\hline
$p_{4C}(x)=x^4+4x^3-16x^2-16x+8$ & $ -6.145... <  -1.179... < 0.374... < 2.950...$ \\
\hline
$p_{4D}(x) = x^4 + 4x^3 -16x^2 -8x +16$ & $ -6.280... < -1.142... < 0.875... < 2.546...$  \\
\hline
$p_{5A}(x) = x^5 -32x^3+64x^2 +32 x -64$ & $ -6.401...  < -1.005... < 0.985... <  2.746... < 3.674...$  \\
\hline
$p_{5B}(x) = x^5-32 x^3 + 64 x^2 +40x -96$ & $ -6.378... < -1.195... < 1.342... < 2.536... < 3.694...$ \\
\hline
\end{tabular}
\end{center}
\caption{The irreducible polynomials of degree at least 2 over $\mathbb{Q}$ that occur in the factorizations of the characteristic polynomials, and their roots} \label{tab4}
\end{table}

\end{document}